
\title{Strong solutions to degenerate SDEs and uniqueness for degenerate Fokker--Planck equations}
\author{Sebastian Grube\thanks{
                  Faculty of Mathematics, Bielefeld University, 33615 Bielefeld, Germany. E-Mail: sgrube@math.uni-bielefeld.de}
        }
\documentclass{article}
\usepackage[utf8]{inputenc}
\usepackage[a4paper, left=2.05cm, right=2.05cm, top=2cm]{geometry}
\usepackage[style=alphabetic, backend=bibtex,giveninits, doi=false,isbn=false,url=false,eprint=false, maxbibnames=200, date=year]{biblatex}
\renewbibmacro{in:}{}
\DeclareNameAlias{author}{family-given}
\usepackage{setspace}
\usepackage{graphicx}
\usepackage{amsmath}
\allowdisplaybreaks
\usepackage{amsfonts}
\usepackage{amssymb}
\usepackage{amsthm}
\usepackage{geometry}
\usepackage{hyperref}
\usepackage{mathrsfs}
\usepackage{bbm}
\usepackage{esint}
\usepackage{enumerate}
\usepackage{enumitem}
\usepackage{bm}

\newcommand{\law}[1]{{\mathcal L}_{#1}}
\newcommand{\LN}{\mathbb N}
\newcommand{\RR}{\mathbb R}

\newcommand{\BBB}{\mathbb B}
\newcommand{\BBBB}{\mathcal B}
\newcommand{\EEEE}{\mathcal E}

\newcommand{\E}{\mathbb{E}}
\newcommand{\FF}{\mathcal F}

\newcommand{\PP}{\mathbb P}

\newcommand{\PPPP}{\mathcal P}
\newcommand{\SSS}{\mathscr S}

\newcommand{\inv}{^{-1}}
\newcommand{\rd}{{\RR^d}}
\newcommand{\scalarproduct}[3][]{\langle #2, #3\rangle_{#1}}
\DeclareMathOperator{\divv}{\mathrm{div}}
\DeclareMathOperator{\M}{\mathrm{M}}

\usepackage{nicefrac}

\usepackage{xifthen}

\newcommand{\norm}[2][]{\ifthenelse{\isempty{#1}}
	{\left\lVert#2\right\rVert}
	{\left\lVert#2\right\rVert_{{#1}}}
}

\newtheorem{theorem}{Theorem}[section]
\newtheorem{lemma}[theorem]{Lemma}

\newtheorem{definition}[theorem]{Definition}
\newtheorem{remark}[theorem]{Remark}
\newtheorem{proposition}[theorem]{Proposition}

\begin{document}
\newpage
\maketitle
\begin{abstract}
We prove the existence of probabilistically strong solutions for large classes of possibly degenerate stochastic differential equations with unbounded and locally Sobolev-regular coefficients, using the restricted Yamada--Watanabe theorem. Our approach relies on existence results for the corresponding Fokker--Planck equation, combined with both novel and existing restricted pathwise uniqueness results for SDEs. Here, restricted pathwise uniqueness means pathwise uniqueness among a subclass of weak solutions to the SDE. Furthermore, we derive new uniqueness results for the Fokker--Planck equation.
\end{abstract}
\vspace{0.5em}
\section{Introduction}\label{section.introduction}
We consider the following  stochastic differential equation (abbreviated as SDE) in $\rd$, $d\in \LN$,  of the form
\begin{align}\label{SDE}
	\mathrm{d}X(t) \notag
	=&\ b(t,X(t))\ \mathrm{d}t  + \sigma(t,X(t))\ \mathrm{d} W(t), \notag \\
	X(0)=&\ X_0,\tag{SDE}
\end{align}
where $t\in [0,T]$, $T\in (0,\infty)$, $(W(t))_{t\in[0,T]}$ is a standard $d_1$-dimensional  $(\FF_t)$-Brownian motion, $d_1 \in \LN$, and $X_0$ an $\RR^d$-valued $\FF_0$-measurable function on some stochastic basis  $(\Omega,\FF,\PP;(\FF_t)_{t\in[0,T]})$, i.e.  a complete, filtered probability space, where $(\FF_t)_{t\in[0,T]}$ is a normal filtration.
Here, we assume that the drift and diffusion coefficient are Borel measurable functions 
\begin{align}\label{SDE.SDEcoefficients}
	b : [0,T]\times\rd  \to \rd,\ \ \
	\sigma: [0,T]\times\rd \to \RR^{d\times d_1}.
\end{align}

Furthermore, we consider the following second order partial differential equation of type 
\begin{align}\label{FPKE}\tag{FPE}
	\partial_t \mu_t &= L^*_t\mu_t \text{ on } (0,T)\times\rd\\
		\left.\mu\right|_{t=0}&=\nu, \notag
\end{align}
where $L^*_t$ is the formal adjoint of the (linear) Kolmogorov operator $L_t$ in $L^2$ given by
\begin{align}\label{SDE.definition.generator}
	L_t := L_t(a,b):= \sum_{i=1}^d b^i(t,x)\partial_{x_i} + \sum_{i,j=1}^d a^{ij}(t,\cdot)\partial_{x_i}\partial_{x_j},
\end{align}
where $b$ is as in \eqref{SDE.SDEcoefficients} and $a: [0,T]\times \rd \to S_+(\rd)$ a Borel measurable function. Here, $S_+(\rd)$ consists of all $d\times d$-dimensional nonnegative symmetric real matrices. Equations of the form \eqref{FPKE} are called \textit{Fokker--Planck(--Kolmogorov) equations}.

Note that \eqref{SDE} and \eqref{FPKE} share a deep connection, which is twofold. Let $(X,W)$ be a probabilistically weak solution to \eqref{SDE}. Then, using It\^o's formula, the solution's time-marginal laws, i.e. $\mu_t:=\PP\circ X(t)\inv, t\in [0,T]$, solve \eqref{FPKE} in the Schwartz-distributional sense, meaning that for all $\varphi \in C_c^\infty(\rd)$
\begin{align}\label{SDE.intro.FPKE.distrSol}
		\int_\rd \varphi(x)\ \mu_t(\mathrm{d}x)
		= \int_\rd \varphi(x)\ \nu(\mathrm{d}x)
		+ \int_0^t \int_\rd (L_s\varphi)(x)\ \mu_s(\mathrm{d}x)\mathrm{d}s  \ \ \forall t \in [0,T].
	\end{align}
	On the other hand, via the \textit{Ambrosio--Figalli--Trevisan superposition principle}, a narrowly continuous curve of probability measures $(\mu_t)_{t\in [0,T]}$ (that is, $[0,T]\ni t \mapsto \int \varphi\ \mathrm{d}\mu_t$ is continuous for all $\varphi \in C_b(\rd)$) solving \eqref{FPKE} for $a = {\sigma\sigma^*\slash 2}$ in the Schwartz-distributional sense can be 'lifted' to a weak solution $(X,W)$ to \eqref{SDE} in the sense that $\law{X(t)}=\mu_t$ for all $t\in [0,T]$. This is true for a large class of coefficients $b$ and $\sigma$ fulfilling a mild integrability condition with respect to $(\mu_t)_{t\in [0,T]}$, e.g. $b^i, a^{ij} \in L^1([0,T]\times\rd;\mu_t\mathrm{d}t)$, $1\leq i,j \leq d$.
The superposition principle appears to be initially introduced by Ambrosio \cite{ambrosio2004transport} in the ODE-case (i.e. $\sigma \equiv 0$) and later on generalised to SDEs by Figalli \cite{figalli2008existence}. Further extensions were made by Trevisan \cite{trevisan2016super}, and Bogachev, R\"ockner and Shaposhnikov \cite{bogachev2021super}.
In particular, the superposition principle serves as an effective tool for obtaining weak solutions to \eqref{SDE}  provided the existence of a solution to \eqref{FPKE} with suitable integrability properties is known. The literature on the existence (and also uniqueness) of solutions to \eqref{FPKE} is quite rich, see e.g. \cite{bogachev2015FPKE} and the references therein.

The aim of this paper is to demonstrate that, under fairly general assumptions on the coefficients $b,\sigma$ and the initial distribution, each probabilistically weak solution $(X,W)$, constructed from a sufficiently regular Schwartz-distributional solution to the corresponding Fokker--Planck equation via the superposition principle, is indeed a functional of the driving Brownian motion $W$, i.e. a probabilistically strong solution. Here, our main contribution in this regard is Theorem \ref{introduction.theorem.mainresult} below.
The core of our approach is the restricted Yamada--Watanabe theorem (see Theorem \ref{theorem.restrictedYamadaWatanabe}), first introduced in \cite{grube2021strong, grube2022thesis}.
This theorem offers a method to prove the existence of strong solutions to \eqref{SDE}, provided that the existence of a weak solution is known and pathwise uniqueness holds among a (potentially small) class of weak solutions.
Such \textit{restricted} pathwise uniqueness results can be found in, e.g.~\cite{roeckner2010weakuniqueness}, \cite{luo2014uniq} and \cite{champagnat2018strong}.
In this paper, we combine the techniques of \cite{roeckner2010weakuniqueness}, and \cite{champagnat2018strong} to prove, in particular, that if $b \in L^1_t(W^{1,1}_{x,\mathrm{loc}})$ and $\sigma \in L^2_t(W^{1,2}_{x,\mathrm{loc}})$, pathwise uniqueness holds among weak solutions to \eqref{SDE} with uniformly bounded time-marginal law densities, see Theorem \ref{SDE.PU.theorem.mainresult} and also Theorem \ref{SDE.PU.theorem.mainresult.abstract} for our most general version. If $b \in L^1_t(W^{1,1}_{x})$ and $\sigma \in L^2_t(W^{1,2}_{x})$ this result has already been obtained in \cite{champagnat2018strong}.
To the best of our knowledge, the case where $b \in L^1_t(W^{1,1}_{x,\mathrm{loc}})$ has not yet been fully explored in the literature. For a comparison with \cite{roeckner2010weakuniqueness} and \cite{luo2014uniq}, we refer to Remark \ref{SDE.PU.remark.OSLcounterexample} and Remark \ref{SDE.PU.remark.SIcounterexample}, respectively.
Moreover, in \cite{roeckner2010weakuniqueness} and \cite{luo2014uniq}, the authors use their restricted pathwise uniqueness result in conjunction with Figalli's superposition principle to derive (restricted) uniqueness results for \eqref{FPKE} under general conditions on the coefficients.
By utilising the recent generalisation of the Ambrosio--Figalli--Trevisan superposition from \cite{bogachev2021super} and the aforementioned  restricted pathwise uniqueness results, we extend the results in \cite{roeckner2010weakuniqueness} to unbounded coefficients and, additionally, obtain new uniqueness results for \eqref{FPKE}, see Section \ref{section.mainResults.uniquenessFPKE}.

The fundamental question of whether probabilistically strong solutions to \eqref{SDE} exist is of course well-studied under various conditions on the coefficients.
When the diffusion coefficient is allowed to be degenerate, the existence of strong solutions was proven in \cite{yamada1971yamada2, kumar2013degenerate, luo2011hoelder, wang2020existence}. Let us, however, point out that in all of these sources the diffusion coefficient is assumed to be continuous in the spacial variable, often even (locally) H\"older continuous.
In \cite{champagnat2018strong}, the authors proved the existence of probabilistically strong solutions to \eqref{SDE} with bounded and, in particular, (globally) Sobolev-regular coefficients, allowing for discontinuous and degenerate diffusion coefficients.
In this work, pairing our procedure, based on the superposition principle and the restricted Yamada--Watanabe theorem, Theorem \ref{SDE.mainresult.abstract}, and the restricted pathwise uniqueness result formulated as Theorem \ref{SDE.PU.theorem.mainresult.abstract}, we essentially generalise \cite[Theorem 2.13 (i)]{champagnat2018strong}.
Indeed, in the case when there exists a family of probability measures $(\mu_t)_{t\in [0,T]}$ solving \eqref{FPKE} in the Schwartz-distributional sense, which can be represented by a locally $p$-integrable function in time and space, we prove the following theorem. It is a direct consequence of Theorem \ref{SDE.mainresult.abstract} combined with Theorem \ref{SDE.PU.theorem.mainresult.abstract}.
\begin{theorem}\label{introduction.theorem.mainresult}
	Let $b,\sigma$ as in \eqref{SDE.SDEcoefficients}, $a :=  \sigma\sigma^*/2$, $p \in [1,\infty), p' \in (1,\infty]$ such that $\nicefrac{1}{p}+\nicefrac{1}{p'}=1$.
	Assume that $\sigma \in L^{2p}([0,T];W^{1,2p}_{\mathrm{loc}}(\rd)), b \in L^p_{loc}([0,T]\times\rd)$.
	Assume that there exists a narrowly continuous curve of probability measures $(\mu_t)_{t\in [0,T]}$ such that $\mu_t = u(t,x)dx$ for some $u \in L^{p'}_{\mathrm{loc}}([0,T]\times\rd)$, which satisfies
	\begin{align*}
		\int_0^T \int_\rd \frac{|\scalarproduct[\rd]{b(t,x)}{x}|+|a(t,x)|}{1+|x|^2}\ \mathrm{d}x\mathrm{d}t < \infty,
	\end{align*} and solves \eqref{FPKE} in the sense of \eqref{SDE.intro.FPKE.distrSol}.
	Furthermore, assume that $b(t,\cdot) \in BV_{\mathrm{loc}}(\rd)$, for a.e. $t\in [0,T]$, and that for every radius $R>0$, there exists a super-linear function $\phi_R$ of modest growth (see Definition \ref{definition.superlinearity}), $\varepsilon_0=\varepsilon_0(R)\in (0,1)$, and a constant $C_{T,R}>0$, which only depends on $T$ and $R$, such that for all $0<\varepsilon<\varepsilon_0$ and all $\gamma \in \rd$,
	\begin{align}\label{introduction.theorem:Meps}
		\sup_{0<\epsilon <\epsilon_0} \frac{\varepsilon\phi(\varepsilon\inv)}{|\ln(\varepsilon)|}\int_0^T \int_{B_R(0)} \M_{\varepsilon\inv}^{R+2} (D_x b(t,\cdot))(x-\gamma) u(t,x)dxdt \leq C,
	\end{align} 
	(for the definition of operators of the form $\M_L^R, L \geq 1,$ see Definition \ref{SDE.PU.definition.MRL}).
	
	Then there exists a probabilistically strong solution to \eqref{SDE}, which is pathwise unique among probabilistically weak solutions with time-marginal law densities $u$.
\end{theorem}
\begin{remark}[Cf. Proposition \ref{SDE.PU.proposition.SobolevMLR}]
	Let $\tilde{p} \in [1,\infty)$, $\tilde{p}' \in (1,\infty]$ such that $\nicefrac{1}{\tilde{p}}+\nicefrac{1}{\tilde{p}'}=1$. Let $\tilde{u} \in L^{\tilde{p}}_{\mathrm{loc}}([0,T]\times\rd)$ and $\tilde{b} \in L^{\tilde{p}}([0,T];W^{1,\tilde{p}'}_{\mathrm{loc}}(\rd))$. Then there exists a super-linear function $\phi$ of modest growth such that \eqref{introduction.theorem:Meps} holds for $b$ replaced by $\tilde{b}$, and $u$ replaced by $\tilde{u}$.
\end{remark}
Let us compare  the above theorem with the main result on the existence of a strong solution to \eqref{SDE} in \cite[Theorem 2.13 (i)]{champagnat2018strong}. In contrast to the assumptions in Champagnat's and Jabin's result, here the coefficients $b$ and $\sigma$ may be unbounded, and their spacial Schwartz-distributional derivatives only need to satisfy, in a sense, a local integrability condition. We would like to emphasise that for $\mu$ being a vector-valued Radon measure on $\rd$, $\M_{L}^{R}(\mu)$ (as in \eqref{introduction.theorem:Meps}) only depends on $\left.\mu\right|_{\BBBB(B_R(0))}$. A similar condition as in \eqref{introduction.theorem:Meps} was first introduced in \cite[Theorem 2.13 (i)]{champagnat2018strong} with respect to the operators $\M_{L}\equiv \M_{L}^{\infty}, L\geq 1$. We note that with $\M_{L}^{R}, L\geq 1,$ we 'localise' the operators $\M_{L}, L\geq 1,$ in this work.
Moreover, our approach is different to the one in \cite[Theorem 2.13 (i)]{champagnat2018strong}. Champagnat's and Jabin's result requires sufficiently nice approximations of the coefficients of \eqref{SDE}, as well as regular and, in a sense, 'well-behaved' solutions to the corresponding (approximating) Fokker--Planck equations, converging in weak sense to a solution to \eqref{FPKE}, see \cite[Theorem 2.13 (i)]{champagnat2018strong} for details. In the above theorem this is not needed; it is more direct in this sense and merely relies on the existence of a sufficiently nice solution to \eqref{FPKE}, and a (restricted) pathwise uniqueness result for \eqref{SDE}.

In Section \ref{section.mainResults.strongSolutionsSDE}, we provide large classes of coefficients $b$ and $\sigma$ for which a probabilistically strong solution to \eqref{SDE} exists. These results rely on results on the existence of sufficiently regular solutions to \eqref{FPKE}, so that the above theorem may be applied (see also Section \ref{section.SDE.applications.SDE.1}).
Here, we would like to stress that such $b$ and $\sigma$ are neither generically spacially continuous, locally bounded, nor is such $\sigma$ necessarily bounded away from zero on sets of positive Lebesgue measure, see Remark \ref{SDE.application.SDE.remark.ams} for a concrete example.
To the best of our knowledge, our results do not seem to be covered by the literature so far.

Moreover, we refer to \cite{grube2021strong, grube2022strongTime, grube2024strongDegenerate}, where the restricted Yamada--Watanabe theorem was used to prove the existence of probabilistically strong solutions for a large class of McKean--Vlasov SDEs with coefficients of Nemytskii-type. Such coefficients structurally lack usual continuity properties in their measure variable, e.g. with respect to the topology induced by the narrow convergence of probability measures. This is in contrast to large parts of the literature, where coefficients with such continuity properties are considered (see \cite{delarue2018mckean} and the references therein).\\

This paper is structured as follows.
In Section \ref{section.solutions}, we introduce a precise notion of solution for \eqref{SDE} and \eqref{FPKE}.
In Section \ref{section.procedure}, we recall the restricted Yamada--Watanabe theorem from \cite{grube2021strong, grube2022thesis} and recall the superposition principle from \cite{bogachev2021super}. We then use these results in order to create a procedure on how to obtain probabilistically strong solutions on the basis of Schwartz-distributional solutions to the corresponding Fokker--Planck equation.
In Section \ref{section.mainResults}, we formulate our main results; in Subsection \ref{section.mainResults.strongSolutionsSDE}, we present new results regarding the existence of probabilistically strong solutions to degenerate \eqref{SDE} and, in Subsection \ref{section.mainResults.uniquenessFPKE}, we present new uniqueness results for \eqref{FPKE}, which are obtained via the technique of \cite{roeckner2010weakuniqueness}.
In Section \ref{section.SDE.pathwiseUniqueness}, we prove a generalisation of the restricted pathwise result from \cite{champagnat2018strong} on the basis of the techniques in \cite{champagnat2018strong} and \cite{roeckner2010weakuniqueness}. 
	This result is used in the subsequent sections.
In Section \ref{section.SDE.applications.SDE}, we prove the main results from Section \ref{section.mainResults.strongSolutionsSDE}.
In Section \ref{section.SDE.applications.FPKE}, we prove the main results from Section \ref{section.mainResults.uniquenessFPKE}.

 This paper emerged from the authors PhD thesis \cite[Chapter 2]{grube2022thesis}.
\subsection*{Notation}
Throughout this paper, we use the following notation, which resembles the one in \cite{grube2024strongDegenerate}. 

The $d$-dimensional real space is denoted by $\rd$ and considered with the usual scalar product $\scalarproduct[\rd]{\cdot}{\cdot}$ and euclidean norm $|\cdot|_\rd$. We suppress the norm's dependence on the dimension $d$, i.e. we write $|\cdot|_\rd=|\cdot|$, if $d$ can be inferred from the context.
For a topological space $(T,\tau)$, the Borel-$\sigma$-algebra is denoted by $\BBBB(\tau)$.
By $\PPPP(T)$ we denote the set of all probability measures on $(T,\BBBB(\tau))$. If $T=\rd$, $\PPPP_0(\rd)$ shall denote the subset of $\PPPP(\rd)$ whose elements have a density with respect to Lebesgue measure on $\rd$.
Let $(\Omega,\FF,\PP)$ be a probability space and $(\Omega',\FF')$ a measurable space. Let $X:\Omega \to \Omega'$ be an $\FF\slash\FF'$-measurable function. Then $\law{X}:=\PP\circ X\inv$ denotes the law of $X$.
Let $n,m \in \LN$. The set of all $\RR^{n\times m}$-valued Radon measures is denoted by $\mathcal{M}_{\mathrm{loc}}(\rd;\RR^{n\times m})$. For $\mu \in \mathcal{M}_{\mathrm{loc}}(\rd;\RR^{n\times m})$, $|\mu|$ denotes its variation. If $n=m=1$, then we write $\mathcal{M}_{\mathrm{loc}}(\rd;\RR^{n\times m})=\mathcal{M}_{\mathrm{loc}}(\rd)$.
By $\mathcal{M}_b(\rd;\RR^{n\times m})$ we denote the set of all Radon measures $\mu \in \mathcal{M}_{\mathrm{loc}}(\rd;\RR^{n\times m})$ such that $|\mu|(\rd) <\infty$.
By $\mathcal{M}_+(\rd)$ we denote the space of all nonnegative Borel measures on $\rd$.
 Let $I\subseteq \RR$ be an interval. A family $(\mu_t)_{t\in I}\subseteq \mathcal{M}_{\mathrm{loc}}(\rd)$ is said to be vaguely continuous if for all $\varphi \in C_c(\rd)$ and all $t_0 \in I$
\begin{align}\label{notation:weakVagueContinuous}
	\int_\rd \varphi(x)\ \mu_t(\mathrm{d}x) \to \int_\rd\varphi(x)\ \mu_{t_0}(\mathrm{d}x) \text{, whenever $t\in I, t\to t_0$.}
\end{align}
If $(\mu_t)_{t\in I}\subseteq \mathcal{M}_b(\rd;\RR^{n\times m})$ and \eqref{notation:weakVagueContinuous} holds for all $\varphi \in C_b(\rd)$, then we say that $(\mu_t)_{t\in I}$ is narrowly continuous.
Let $(E,||\cdot||_E)$ be a Banach space. By $C([0,T];E)$ we denote the set of continuous functions from $[0,T]$ to $E$. We always consider $C([0,T];E)$ together with the usual supremum's norm. Furthermore, $C([0,T];E)_0$ shall consist of all $w\in C([0,T];E)$ such that $w(0)=0$. We consider $C([0,T];E)_0$ with the supremum's norm on $C([0,T];E)_0$.
For $t\in [0,T]$, $\pi_t: C([0,T];E) \to E$ denotes the canonical evaluation map at time $t$, i.e. $\pi_t(w):=w(t), w\in C([0,T];E)$. Furthermore, we set $\BBBB_t(C([0,T];E)):= \sigma(\pi_s : s\in [0,t])$ and, correspondingly, $\BBBB_t(C([0,T];E)_0):= \sigma(\pi_s: s\in [0,t])\cap C([0,T];E)_0$.
Moreover, $\PP^W$ denotes the Wiener measure on $(C([0,T];\rd)_0,\BBBB(C([0,T];\rd)_0))$.

Let $(S,\SSS,\eta)$ be a measure space. For $1\leq p \leq \infty$, $(L^p(S;E), \norm[{L^p(S;E)}]{\cdot})$ denotes the usual Bochner space (with usual norm) of strongly measurable $E$-valued functions $f$ on $S$ for which $\norm[E]{f}^p$ is integrable if $1\leq p <\infty$ and usual adaption if $p=\infty$. 
If $S=\RR^n$ and $E=\RR$, we just write $L^p(\RR^n;\RR)=L^p(\RR^n)$. The set of strongly measurable functions on $\RR^n$ with values in $E$ which are locally $p$-integrable in norm on $E$, will be denoted by $L^p_{\mathrm{loc}}(\RR^n;E)$.
Furthermore, $(W^{k,p}(\RR^n), \norm[{W^{1,p}(S;E)}]{\cdot})$ denotes the usual Sobolev space, containing all $L^p(\RR^n)$-functions, whose Schwartz-distributional derivatives up to $k$-th order can be represented by elements in $L^p(\RR^n)$. 
Accordingly, $E$-valued Sobolev functions on $\RR^n$ will be denoted by $W^{k,p}(\RR^n;E)$.
The local analogues of $W^{k,p}(\RR^n;E)$ are denoted by $W_{\mathrm{loc}}^{k,p}(\RR^n;E)$.
Moreover, $BV_{\mathrm{loc}}(\rd;\rd)$ denotes the set of functions which are locally of bounded variation, i.e. it consists of all $f\in L^1_{\mathrm{loc}}(\rd;\rd)$ such that its Schwartz-distributional derivative $Df$ can be represented as an element in $\mathcal{M}_{\mathrm{loc}}(\rd;\RR^{d\times d})$.
In general, throughout this paper the Jacobian $D=D_x$, divergence $\divv=\divv_x$, and the gradient $\nabla=\nabla_x$ are always considered in the Schwartz-distributional sense and with respect to the spacial variable.
Furthermore, for every real-valued function, we set $f^- := -\min (f,0)$ and $f^+ := \max(f,0)$.
Also, we set $s\wedge t := \mathrm{min}(t,s),\ s\vee t := \mathrm{max}(t,s)$ for all $s,t \in \RR$.
\section{Solutions for \eqref{SDE} and \eqref{FPKE}}\label{section.solutions}
\subsection{Solutions for \eqref{SDE}}
Throughout this section, let $\nu \in \PPPP(\rd)$ and $P_\nu \subseteq \PPPP(C([0,T];\rd))$ such that
\begin{align}\label{solutions:setOfMeasuresOnPathSpace}
	P_\nu \subseteq \{Q \in \PPPP(C([0,T];\rd)) : Q \circ \pi_0\inv = \nu \}.
\end{align}

We have the following definitions.
\begin{definition}[$P_\nu$-weak solution]\label{SDE.SDE.definition.weakSolution}
A tuple $(X,W)=(X(t),W(t))_{t\in [0,T]}$ consisting of two $(\FF_t)$-adapted $\RR^d$-valued continuous stochastic processes
 on some given stochastic basis
	$(\Omega,\FF,\PP;(\FF_t)_{t\in [0,T]})$
is called a $P_\nu$-weak solution to \eqref{SDE} if
\begin{enumerate}[label=(\roman*)]
	\item 
		$\int_0^T\int_K |b(t,x)|+|\sigma(t,x)|^2\ \law{X(t)}(\mathrm{d}x)\mathrm{d}t<\infty\ \ \forall K\subset\rd$ compact,
	\item $(W(t))_{t\in [0,T]}$ is a standard $d_1$-dimensional $(\FF_t)$-Brownian motion,
	
	\item the following equality holds $\PP$-a.s.:
		 \begin{align*}
		 	X(t)=X(0) + \int_0^t b(s,X(s))\ \mathrm{d}s + \int_0^t \sigma(s,X(s))\ \mathrm{d}W(s)\ \ \forall t\in[0,T].
		 \end{align*}
	\item $\PP\circ X\inv \in P_\nu$. (In particular, $\PP\circ X(0)\inv = \nu$.)
\end{enumerate}

In the case $P_\nu=\{Q \in \PPPP(C([0,T];\rd)): Q\circ \pi_0\inv = \nu\}$, a $P_\nu$-weak solution is also called a weak solution with initial distribution $\nu$.
\end{definition}
It is well known that, loosely speaking, the notion of weak solution is equivalent to the notion of solution to the associated martingale problem. In particular, this allows the study of weak solutions to \eqref{SDE} from a rather analytic point of view. The martingale approach was initiated in \cite{stroock1969diffusion1,stroock1969diffusion2} by Stroock and Varadhan and the theory was issued on a large scale in the authors' book \cite{stroock2006multi}. The following definition of the martingale problem is essentially taken from \cite{bogachev2021super}.
\begin{definition}[solution to the martingale problem]\label{SDE.definition.martingaleSolution}
Let $\nu \in \PPPP(\rd)$. We say that a probability measure $Q \in \PPPP(C([0,T];\rd))$ is a solution to the martingale problem associated with the operator $L=L(a,b)$ as in \eqref{SDE.definition.generator} starting at $\nu$ if
\begin{enumerate}[label=(\roman*)]
	\item $\int_0^T\int_{K}(|b(t,x)| + |a(t,x)|)\  (Q\circ\pi_t\inv)(\mathrm{d}x)\mathrm{d}t< \infty$\ \ $\forall K\subseteq \rd$ compact,
	\item for each $\varphi \in C_c^\infty(\rd)$, the family of random variables
		\begin{align*}
			t\mapsto \varphi\circ\pi_t - \varphi\circ\pi_0 - \int_0^t (L_s \varphi)(\pi_s) \mathrm{d}s
		\end{align*}
		is an $(\FF_t)$-martingale with respect to $(C([0,T];\rd), \BBBB(C([0,T];\rd)), Q)$, where ${\FF_t := \BBBB_t(C([0,T];\rd))}$, $t\in [0,T]$.
		\item $Q\circ \pi_0\inv = \nu$ on $(\rd, \BBBB(\rd))$.
\end{enumerate}

\end{definition}
\begin{definition}[$P_\nu$-weak uniqueness]
We say that $P_\nu$-weak uniqueness holds for \eqref{SDE}, if every two $P_\nu$-weak solutions $(X,W)$, $(Y,W')$ on stochastic bases $(\Omega,\FF,\PP;(\FF_t)_{t\in [0,T]})$ and $(\Omega',\FF',\PP';(\FF'_t)_{t\in [0,T]})$, respectively, have the same law, i.e.
\begin{align*}
	\PP\circ X\inv = \PP'\circ Y\inv
\end{align*}
(as measures on $\BBBB(C([0,T];\rd))$).
In the case $P_\nu=\{Q \in \PPPP(C([0,T];\rd)): Q\circ \pi_0\inv = \nu\}$, $P_\nu$-weak uniqueness is also called weak uniqueness among weak solutions with initial distribution $\nu$.
\end{definition}
The following result shows that the notion of a weak solution to \eqref{SDE} is indeed equivalent to the notion of a solution to the just introduced martingale problem with respect to $L({\sigma\sigma^*\slash 2},b)$ in terms of existence and uniqueness.

The following proposition can be found in \cite[Proposition 2.2.3]{grube2022thesis} and it is based on \cite[Proposition 4.11, Prob. 4.13, Corollary 4.8, Corollary 4.9]{karatzas91}.

\begin{proposition}[MP $\rightleftarrows$ weak solutions]\label{SDE.proposition.martingaleSolutionEQweaksolution}
	Let $\nu \in \PPPP(\rd)$ and $b,\sigma$ as in \eqref{SDE.SDEcoefficients}.
	We abbreviate $L=L(\sigma\sigma^*\slash 2,b)$.
	There exists a weak solution $(X,W)$ to \eqref{SDE} with initial distribution $\nu$ if and only if there exists a solution to the martingale problem $Q$ associated with the operator $L$ starting at $\nu$. In this case, the relation $\law{X}=Q$ can be chosen to hold.
	Moreover, weak uniqueness holds for \eqref{SDE} among weak solutions with initial distribution $\nu$ if and only if the martingale problem associated with the operator $L$ starting at $\nu$ has a unique solution.
\end{proposition}
\begin{definition}[$P_\nu$-pathwise uniqueness]
We say that $P_\nu$-pathwise uniqueness holds for \eqref{SDE}, if for every two $P_\nu$-weak solutions $(X,W)$, $(Y,W)$ on a common stochastic basis $(\Omega,\FF,\PP;(\FF_t)_{t\in [0,T]})$ with a common standard $d_1$-dimensional $(\FF_t)$-Brownian motion $(W(t))_{t\in [0,T]}$, 
$$\text{$X(0)=Y(0)$ $\PP$-a.s. \text{implies} $X(t)=Y(t)$ for all $t\in [0,T]\ \PP$-a.s.}$$
\end{definition}
Let $\tilde{\EEEE}_{\nu}$ be the set of all maps $F_{\nu} : \rd\times C([0,T];\rd)_0 \to C([0,T];\rd)$ which are $\overline{\BBBB(\rd)\otimes \BBBB(C([0,T];\rd)_0)}^{{\nu}\otimes \PP^W}\slash\BBBB(\BBB)$-measurable.
Here $\overline{\BBBB(\rd)\otimes \BBBB(C([0,T];\rd)_0)}^{{\nu}\otimes \PP^W}$ denotes the completion of $\BBBB(\rd)\otimes \BBBB(C([0,T];\rd)_0)$ with respect to the measure ${\nu}\otimes \PP^W$.
\begin{definition}[$P_\nu$-strong solution]\label{SDE.definition.strongSolution}
The equation \eqref{SDE} has a $P_\nu$-strong solution, if there exists $F_{\nu}\in \tilde{\EEEE}_{\nu}$ such that for $\nu$-a.e. $x\in \rd$, $F_{\nu}(x,\cdot)$ is $\overline{\BBBB_t(C([0,T];\rd)_0)}^{\PP^W}\slash \BBBB_t(C([0,T];\rd))$-measurable for every $t\in[0,T]$, and for every standard $d_1$-dimensional $(\FF_t)$-Brownian motion $(W(t))_{t\in [0,T]}$ on a stochastic basis $(\Omega,\FF,\PP;(\FF_t)_{t\in [0,T]})$ and every $\FF_0\slash \BBBB(\rd)$-measurable function $\xi: \Omega \to \rd$, with $\PP\circ \xi\inv={\nu}$, one has that
	$(F_{\nu}(\xi,{W}),W)$ is a $P_{\nu}$-weak solution to \eqref{SDE} with $X(0)=\xi$ $\PP$-a.s.
	Here, $\overline{\BBBB_t(C([0,T];\rd)_0)}^{\PP^W}$ denotes the completion with respect to $\PP^W$ in $\BBBB(C([0,T];\rd)_0)$.
\end{definition}
\begin{definition}[unique $P_{\nu}$-strong solution]\label{SDE.definition.uniqStrongSolution}
	The equation \eqref{SDE} has a \textit{unique $P_{\nu}$-strong solution}, if there exists a function $F_{\nu}\in \tilde{\EEEE}_{\nu}$ satisfying the adaptedness condition in Definition \ref{SDE.definition.strongSolution} and if the following two conditions are satisfied.
	\begin{enumerate}[label=(\roman*)]
		\item \label{SDE.definition.uniqStrongSolution:1}
		For every standard $d_1$-dimensional $(\FF_t)$-Brownian motion $(W(t))_{t\in [0,T]}$ on a stochastic basis $(\Omega,\FF,\PP;(\FF_t)_{t\in [0,T]})$ and every $\FF_0\slash \BBBB(\rd)$-measurable $\xi:\Omega \to \rd$ with $\law{\xi}=\nu$, $(F_{\nu}(\xi,W),W)$ is a $P_{\nu}$-weak solution to \eqref{SDE}.
		\item \label{SDE.definition.uniqStrongSolution:2} For every $P_{\nu}$-weak solution $(X,W)$ to \eqref{SDE} on a stochastic basis $(\Omega,\FF,\PP;(\FF_t)_{t\in [0,T]})$ we have
		\begin{align*}
			X=F_{\nu}(X(0),W)\ \PP\text{-a.e.}
		\end{align*}
		\end{enumerate}
\end{definition}
\begin{remark}\label{SDE.definition.remark.uniqStrongSolutionImpliesWeakUniq}
Let $(X,W)$ be a $P_{\nu}$-weak solution to \eqref{SDE} on a stochastic basis $(\Omega,\FF,\PP;(\FF_t)_{t\in [0,T]})$.
	Since $X(0)$ and $W$ are $\PP$-independent, we have
	\begin{align*}
		\PP\circ(X(0),W)\inv=\nu\otimes \PP^W.
	\end{align*}
	In particular, the existence of a unique $P_{\nu}$-strong solution to \eqref{SDE} implies that $P_{\nu}$-weak uniqueness holds for \eqref{SDE}.
\end{remark}

\subsection{Solutions for \eqref{FPKE}}
In Section \ref{section.introduction}, we briefly introduced the meaning of a solution to \eqref{FPKE}.
In this section, we rigorously define our solution framework.

In general, the Fokker--Planck equation \eqref{FPKE} is understood as an equation for Radon measures on ${(0,T)\times \rd}$ in the Schwartz-distributional sense, see \cite[(6.1.2) and (6.1.3)]{bogachev2015FPKE}.
We will just consider solutions of the form $\mu = \mu_t\mathrm{d}t$. Solutions are then meant in the sense of Definition \ref{FPKE.definition.solution.measure}. Note that $\mu_t$, $t\in [0,T]$, is then only determined $\mathrm{d}t$-a.e. We will consider the case that $(\mu_t)_{t \in [0,T]}$ is a vaguely continuous curve of Radon measures on $\rd$. In fact, in many cases, the $\mathrm{d}t$-a.e. determined family of measures $(\mu_t)_{t\in [0,T]}$ has a unique vaguely continuous version, see \cite[Lemma 2.3]{rehmeier2022nonlinearFlow}. 
In view of the connection between \eqref{SDE} and \eqref{FPKE}, this consideration is very natural. Indeed, if $\mu_t= \law{X_t}$, $t\in [0,T]$, where $(X,W)$ is a weak solution to \eqref{SDE}, $(\mu_t)_{t \in [0,T]}$ is automatically narrowly continuous due to the continuity of $t\mapsto X(t)$.

The following definition is consistent with Definition \ref{FPKE.definition.solution.measure}, see Remark \ref{FPKE.definition.solution.measure:remark}.

\begin{definition}[solution to \eqref{FPKE}]\label{FPKE.definition.solution}
	Let $\nu \in \mathcal{M}_{\mathrm{loc}}(\rd)$. A vaguely continuous curve $(\mu_t)_{t\in [0,T]}\subseteq \mathcal{M}_{\mathrm{loc}}(\rd)$ is called a \textit{solution} to the Cauchy problem \eqref{FPKE} with $\left.\mu\right|_{t=0}=\nu$ if $t\mapsto |\mu_t|(B) \in L^1([0,T])$, for every Borel measurable precompact set $B\subset\rd$,
	 and
	\begin{enumerate}[label=(\roman*)]
		\item\label{SDE.definition.FPKEsolution.measure.cond1} $b \in L^1_{\mathrm{loc}}([0,T]\times\rd;\rd;\mu_t\mathrm{d}t)$, $a \in L^1_{\mathrm{loc}}([0,T]\times\rd;\RR^{d\times d};\mu_t\mathrm{d}t)$,
  \item\label{SDE.definition.FPKEsolution.measure.cond2} for each $t \in [0,T]$ and for each $\varphi \in C_c^\infty(\rd)$, we have
			\begin{align}\label{FPKE.definition.solution:1}
				\int_\rd \varphi(x)\ \mu_t(\mathrm{d}x) = \int_\rd \varphi(x)\ \nu(\mathrm{d}x) + \int_0^t \int_\rd 
				(L_s\varphi)(x)\
				\mu_s(\mathrm{d}x)\mathrm{d}s.
			\end{align}
	\end{enumerate}
Moreover, if for all $t\in [0,T]$, $\mu_t$ is absolutely continuous with respect to Lebesgue measure with Radon--Nikodym density $u_t\in L^1(\rd)$,
	then the family $(u_t)_{t\in [0,T]}$ is called an $L^1$-solution to \eqref{FPKE}.
\end{definition}
\begin{definition}[probability solution to \eqref{FPKE}]\label{SDE.definition.FPKEsolution.probability}
	A solution $(\mu_t)_{t\in [0,T]}$ to \eqref{FPKE} is called a probability solution to \eqref{FPKE} if $(\mu_t)_{t\in [0,T]} \subseteq \PPPP(\rd)$.
	Moreover, if  $(u_t)_{t\in [0,T]}$ is an $L^1$-solution to \eqref{FPKE} such that $u_t \in \PPPP_0(\rd)$ for all $t\in [0,T]$, then $(u_t)_{t\in [0,T]}$ is called a $\PPPP_0$-solution to \eqref{FPKE}.
\end{definition}

\section{The procedure}\label{section.procedure}
\subsection*{The restricted Yamada--Watanabe theorem}\label{section.restrictedYamadaWatanabeTheorem}
We recall the restricted Yamada--Watanabe theorem from \cite{grube2021strong} (see \cite{grube2022thesis} for a general version formulated in the variational framework for SPDEs), which is a modification of the original Yamada--Watanabe theorem. In particular, it has the advantage that one can conclude the existence of a probabilistically strong solution to \eqref{SDE} under a relaxed pathwise uniqueness condition compared to the original Yamada--Watanabe theorem.

The following theorem is a variant of \cite[Theorem 3.3]{grube2021strong}. For the most general formulation, we refer to \cite[Theorem 1.3.1]{grube2022thesis}. 
\begin{theorem}
\label{theorem.restrictedYamadaWatanabe}
Let $\nu \in \PPPP(\rd)$ and let $P_\nu$ as in \eqref{solutions:setOfMeasuresOnPathSpace}. Suppose $(\mu_t)_{t\in [0,T]}$ is a probability solution to \eqref{FPKE} such that $\left.\mu\right|_{t=0}=\nu$.
Then the following statements regarding \eqref{SDE} are equivalent. 
\begin{enumerate}[label=(\roman*)]
	\item There exists a $P_\nu$-weak solution and $P_\nu$-pathwise uniqueness holds.
	\item There exists a unique $P_\nu$-strong solution to \eqref{SDE}.
\end{enumerate}
\end{theorem}
\subsection*{A recent superposition principle}
The superposition principle is a tool to 'lift' a probability solution $(\mu_t)_{t\in [0,T]}$ to \eqref{FPKE} to a solution $Q \in \PPPP(C([0,T];\rd)$ to the corresponding martingale problem or, equivalently, to a probabilistically weak solution $(X,W)$ to \eqref{SDE} in such a way that for all $t\in [0,T]$
\begin{align*}
	Q\circ \pi_t\inv = \mu_t\ \ \  \text{or}\ \ \  \law{X(t)} = \mu_t \text{, respectively.}
\end{align*}
The most up-to-date superposition principle in the stochastic case is due to Bogachev, R\"ockner, and Schaposhnikov \cite{bogachev2021super}. The precise statement is as follows.
\begin{theorem}[{\text{\cite[Theorem 1.1]{bogachev2021super}}}]\label{SDE.theorem.superpositionRoe}
	Let $\nu \in \PPPP(\rd)$. Assume that $(\mu_t)_{t\in [0,T]}$ is a probability solution to \eqref{FPKE} with $\left.\mu\right|_{t=0}=\nu$ with respect to the Kolmogorov operator $L(a,b)$, such that
	\begin{align}\label{SDE.superposition.conditionRoe}\tag{SP}
		\int_0^T \int_\rd \frac{|\scalarproduct[\rd]{b(t,x)}{x}|+|a(t,x)|}{1+|x|^2}\ \mu_t(\mathrm{d}x)\mathrm{d}t <+\infty.
\end{align}
	Then there exists a solution $Q \in \PPPP(C([0,T];\rd))$ to the martingale problem for the operator $L(a,b)$ starting in $\nu$ such that $Q\circ\pi_t\inv = \mu_t$ for all $t\in [0,T]$.
\end{theorem}	
We also refer to the recent work by R\"ockner, Xie, and Zhang, which relates solutions to non-local FPKEs with solutions to SDEs with jumps \cite{roeckner2020levy}.
For results in infinite dimensions, we refer to Trevisan's \cite{trevisan2014thesis} and Dieckmann's \cite{dieckmann2020thesis} PhD theses.
\subsection*{From a probability solution to \eqref{FPKE} to a probabilisitically strong solution to \eqref{SDE}}
\label{section.SDE.abtractMainResult}
The superposition principle can directly be combined with the restricted Yamada--Watanabe theorem in order to obtain the following theorem.
\begin{theorem}\label{SDE.mainresult.abstract}
Let $b,\sigma$ as in \eqref{SDE.SDEcoefficients}, and  $a:={\sigma\sigma^*\slash 2}$. Let $\nu \in \PPPP(\rd)$ and $P_\nu$ as in \eqref{solutions:setOfMeasuresOnPathSpace}. Assume that $(\mu_t)_{t\in [0,T]}$ is a probability solution to \eqref{FPKE} with $\left.\mu\right|_{t=0}=\nu$ with respect to the Kolmogorov operator $L(a,b)$. Assume that \eqref{SDE.superposition.conditionRoe} holds and that 
\begin{align*}
	P^{(\mu_t)} := \{Q \in \PPPP(C([0,T];\rd)) : Q \circ \pi_t\inv = \mu_t \ \forall t\in [0,T]\}\subseteq P_{\nu}.
\end{align*}
Then the following statements are equivalent. 
\begin{enumerate}[label=(\roman*)]
	\item $P_{\nu}$-pathwise uniqueness holds for \eqref{SDE}.
	\item There exists a unique $P_{\nu}$-strong solution to \eqref{SDE}.
\end{enumerate}
\end{theorem}
\begin{proof}
We only prove '(i)$\implies$ (ii)'; the other direction is obvious.
Combining Theorem \ref{SDE.theorem.superpositionRoe} and Proposition \ref{SDE.proposition.martingaleSolutionEQweaksolution}, 
we can construct a $P^{(\mu_t)}$-weak solution to \eqref{SDE}.
Therefore, the assertion follows directly from Theorem \ref{theorem.restrictedYamadaWatanabe}.
\end{proof}

\section{Main results}\label{section.mainResults}
We will provide \textit{two types} of applications of Theorem \ref{SDE.mainresult.abstract}.
First, we will utilise it to show the existence of probabilistically strong solutions to \eqref{SDE} for large classes of coefficients and initial distributions.
Second, we will extend the restricted uniqueness result for \eqref{FPKE} from \cite{roeckner2010weakuniqueness} to unbounded coefficients $b,a$ utilising (new) restricted pathwise uniqueness results discussed in Section \ref{section.SDE.pathwiseUniqueness}.
 \subsection{Probabilistically strong solutions to degenerate SDEs}\label{section.mainResults.strongSolutionsSDE}
Relying on the existence results for \eqref{FPKE} from \cite{lebris2008existence,figalli2008existence,bogachev2015FPKE} and the restricted pathwise uniqueness results from Section \ref{section.SDE.pathwiseUniqueness}, we are able to provide large classes of coefficients $b,\sigma$ for which we can prove the existence of probabilistically strong solution to \eqref{SDE} via Theorem \ref{SDE.mainresult.abstract}. To the best of our knowledge, these results are new. The proofs can be found in Section \ref{section.SDE.applications.SDE}.
\vspace{1em}

If $b^i, a^{ij}:=(\sigma\sigma^*)^{ij}\slash 2\in L^1_{\mathrm{loc}}([0,T]\times\rd)$, $i,j \in \{1,...,d\}$, we define the following Schwartz distribution $\beta$ through
\begin{align}\label{SDE.results:beta}
	\beta^i := b^i-\sum_{j=1}^d \partial_{x_j}a^{ij},\ \ i \in \{1,...,d\}.
\end{align}
Furthermore, we introduce the following notation.
Let $p'\in [1,\infty]$ and $\nu \in \PPPP(\rd)$. We define $P^{p'}_\nu$ and $P^{p' , \mathrm{loc}}_\nu$ via
\begin{align*}
	P^{p' (, \mathrm{loc})}_\nu:=\left\{ Q \in \PPPP(C([0,T];\rd)) :\  \exists \rho\in L^{p'}_{(\mathrm{loc})}([0,T]\times\rd) \text{ such that }\right.
	\\\left.Q \circ \pi_t\inv = \rho(t,x)\mathrm{d}x \text{ for a.e. } t\in [0,T], Q\circ\pi_0\inv=\nu  \vphantom{L^{p'}_{(\mathrm{loc})}([0,T]\times\rd)}\right\}.
\end{align*}
In the following, we will always assume that $P^{p'}_\nu\neq \emptyset$ or $P^{p' , \mathrm{loc}}_\nu \neq \emptyset$, respectively. This might, for example, implicitly put restrictions on the choice of $\nu$. Note that in the case $P^{\infty}_{\nu}\neq \emptyset$, $\nu$ needs to have a bounded density with respect to Lebesgue measure. Indeed, if $Q\in P^{\infty}_{\nu}$ the narrow continuity of the curve of measures $(Q\circ\pi_t\inv)_{t\in [0,T]}$ yields that for all $\varphi \in C_c^\infty(\rd)$
\begin{align*}
	\left|\int_\rd \varphi\ \mathrm{d}\nu\right| \leq \norm[L^1(\rd)]{\varphi}
	\norm[{L^\infty([0,T]\times\rd)}]{\rho}.
\end{align*}
Hence, the linear functional $C_c^\infty(\rd)\ni\varphi \mapsto \int_\rd \varphi \ \mathrm{d}\nu$ can be uniquely extended to a bounded linear functional on $L^1(\rd)$ and is therefore represented by a function in $L^\infty(\rd)$. A similar argument yields that in the same case $Q\circ \pi_t\inv$ is absolutely continuous with respect to Lebesgue measure for \textit{all} $t\in [0,T]$ and $\frac{\mathrm{d}(Q \circ \pi_t\inv)}{\mathrm{d}x}\in L^\infty(\rd)$ for \textit{all} $t\in [0,T]$.

Our main results are the following. 
In the case of bounded coefficients, we obtain the following result based on \cite{figalli2008existence}.
\begin{theorem}[main result I.a]\label{SDE.application.SDE.theorem.figalli}
	Let $b\in L^\infty([0,T]\times\rd;\rd), \sigma \in L^\infty([0,T]\times\rd;\RR^{d\times d_1})$ such that $|D_x b| \in L^1_{\mathrm{loc}}([0,T]\times\rd)$, $|D_x \sigma| \in L^2_{\mathrm{loc}}([0,T]\times\rd)$, and
	$\divv \beta \in L^1_{\mathrm{loc}}([0,T]\times\rd)$ 
	with 
	$(\divv \beta)^- \in L^1([0,T];L^\infty(\rd))$.
Let $v\in (\PPPP_0\cap L^\infty)(\rd)$. 
Then there exists a unique $P^{\infty,\mathrm{loc}}_{v}$-strong solution to \eqref{SDE}.
\end{theorem}
In the case of unbounded coefficients we obtain the following two results.
The first one is based on \cite{lebris2008existence} and the second one on \cite{bogachev2015FPKE}.
\begin{theorem}[main result I.b]\label{SDE.application.SDE.theorem.lebrislion}
Let $b \in L^1([0,T]\times\rd;\rd), \sigma \in L^2([0,T]\times\rd;\RR^{d\times d_1})$ such that $|D_x b| \in L^1_{\mathrm{loc}}([0,T]\times\rd), |D_x \sigma| \in L^2_{\mathrm{loc}}([0,T]\times\rd)$, $\divv \beta \in L^1([0,T];L^\infty(\rd))$		
and for each $i \in \{1,...,d\}$
		\begin{align*}
		\sum_{j=1}^d\partial_{j} (\sigma\sigma^*)^{ij} \in L^1([0,T];W^{1,1}_{\mathrm{loc}}(\rd)), \frac{\sum_{j=1}^d\partial_{j} (\sigma\sigma^*)^{ij}}{1+|x|} \in\left(L^1([0,T];L^1(\rd))+ L^1([0,T];L^\infty(\rd))\right).
		\end{align*} 
		Let $v \in (\PPPP_0\cap L^\infty)(\rd)$.
		Then there exists a unique $P^{\infty}_{v}$-strong solution to \eqref{SDE}.
\end{theorem}
\begin{theorem}[main result I.b']\label{SDE.application.SDE.theorem.ams}
	Let $p,p' \in (1,\infty)$ such that $\nicefrac{1}{p}+\nicefrac{1}{p'}=1$. Let $b \in L^p([0,T]\times\rd;\rd), \sigma \in L^{2p}([0,T]\times\rd;\RR^{d\times d_1})$ with $\divv \beta \in L^1_{\mathrm{loc}}([0,T]\times\rd)$ and $(\divv \beta)^- \in L^1([0,T];L^\infty(\rd))$ such that, for every $R>0$,
	$\sup_{t\in [0,T]}\norm[{L^{2p}(B_R(0))}]{|D \sigma_t|} <\infty$	and such that there exists ${f_R \in L^p([0,T]\times B_R(0))}$ with
	\begin{align*}
		\scalarproduct[\rd]{x-y}{b(t,x)-b(t,y)}\leq \left(f_R(t,x) + f_R(t,y)\right)|x-y|^2,
	\end{align*}
	for a.e. $(t,x,y) \in [0,T]\times B_R(0)\times B_R(0)$.
		Let $v \in (\PPPP_0\cap L^{p'})(\rd)$.
		Then there exists a unique $P^{p'}_{v}$-strong solution to \eqref{SDE}.
\end{theorem}
\begin{remark}\label{SDE.application.SDE.remark.ams}
	We would like to emphasise that the conditions in Theorem \ref{SDE.application.SDE.theorem.ams} do not generically imply that $\sigma$ is locally bounded, continuous or bounded away from zero on sets of positive Lebesgue measure.
	Indeed, let $d\geq 5$, $\alpha \in \left(0, \frac{d-4}{2(d-2)}\wedge \frac{1}{2}\right)$ and $f\in L^\infty((0,T))$. Then it is easy to check that  for $b\equiv 0$ and $\sigma(t,x):= f(t)|x|^{\alpha(2-d)}\mathbbm{1}_{d\times d}, (t,x) \in [0,T]\times\rd,$ there exists some $p\in (1,\infty)$ such that the conditions of Theorem \ref{SDE.application.SDE.theorem.ams} are satisfied. 
	\end{remark}
\subsection{Uniqueness results for \eqref{FPKE} obtained with the technique of \cite{roeckner2010weakuniqueness}}\label{section.mainResults.uniquenessFPKE}
In \cite{roeckner2010weakuniqueness}, the authors proved a restricted uniqueness result for linear Fokker--Planck--Kolmogorov equations with bounded coefficients, based on the superposition principle and a restricted pathwise uniqueness result for SDEs.
Since in \cite{roeckner2010weakuniqueness} the authors needed to use the superposition principle by Figalli at that time, their result has automatically been restricted to bounded coefficients $b, \sigma$. Their pathwise uniqueness argument, however, works also in the case of an unbounded drift coefficient $b$ and diffusion coefficient $\sigma$ (cf. Theorem \ref{SDE.PU.theorem.rz10}). So, using a superposition principle result for not necessarily bounded coefficients such as in \cite{trevisan2016super} or \cite{bogachev2021super} we can obtain a more general result.
Furthermore, we complement their result by a local Sobolev condition on the coefficients, which is not covered by the condition \textup{(MC$_{p}^{b,\sigma}$)} below, see Remark \ref{SDE.PU.remark.OSLcounterexample}. This local Sobolev condition stems from the new pathwise uniqueness result in Theorem \ref{SDE.PU.theorem.mainresult}, whose proof is based on the techniques employed in \cite{champagnat2018strong} and \cite{roeckner2010weakuniqueness}.

We introduce the following monotonicity condition for the coefficients $b$ and $\sigma$ of \eqref{SDE}.
\begin{enumerate}[label=\textup{(MC$_{p}^{b,\sigma}$)}, leftmargin=1.6cm]
	\item Assume that for any radius $R>0$ there exists a function $f_R \in L^{p}([0,T]\times B_R(0))$ such that for a.e. $(t,x,y) \in [0,T]\times B_R(0)\times B_R(0)$\label{SDE.PU.condition.monotonicity}
	\begin{align}\tag{OSL}\label{SDE.PU.condition.oneSidedLipschitz:eq}
		2\scalarproduct[\rd]{x-y}{{b}(t,x)-{b}(t,y)}&+|{\sigma}(t,x)-{\sigma}(t,y)|^2 
		\leq (f_R(t,x)+f_R(t,y)) |x-y|^2.\notag
	\end{align}
\end{enumerate}
We have the following result.
\begin{theorem}[uniqueness among solutions to \eqref{FPKE}]\label{SDE.application.FPKE.theorem.uniqueness}
	Let $p,\bar{p},q,\bar{q} \in [1,\infty]$ such that ${\nicefrac{1}{p}+\nicefrac{1}{q}=1}$ and ${\nicefrac{1}{\bar{p}}+\nicefrac{1}{\bar{q}}=1}$. Assume that
		$b \in L^{\bar{p}}([0,T]\times\rd;\rd), \sigma \in L^{2\bar{p}}([0,T]\times\rd;\RR^{d\times d_1})$ and \ref{SDE.PU.condition.monotonicity} holds.
		In the case $p=1$, the assumption \textup{\textup{(MC$_{1}^{b,\sigma}$)}} can be replaced by ${b \in L^1([0,T];W^{1,1}_{\mathrm{loc}}(\rd;\rd))}, \sigma \in L^2([0,T];W^{1,2}_{\mathrm{loc}}(\rd;\RR^{d\times d_1}))$.
	
	Let $\nu \in \mathcal{M}_b(\rd)\cap \mathcal{M}_+(\rd)$. Then 
	there exists at most one solution to \eqref{FPKE} with $a:={\sigma\sigma^*\slash 2}$  in $\mathscr{L}^{q, \bar{q}}_\nu$, where
		$\mathscr{L}^{q, \bar{q}}_\nu$ is defined as the set of narrowly continuous curves $(\mu_t)_{t\in [0,T]}\subseteq \mathcal{M}_b(\rd)\cap \mathcal{M}_+(\rd)$ such that $\mu_0 = \nu$ and
		$\mu\in L^{\bar{q}}([0,T]\times\rd)\cap L^{q}_{\mathrm{loc}}([0,T]\times\rd)$.
\end{theorem}
\begin{remark}
In \cite{luo2014uniq}, Luo extended the weak uniqueness result of \cite[Theorem 1.1]{roeckner2010weakuniqueness} (i.e. Theorem \ref{SDE.application.FPKE.theorem.uniqueness} in the case $p=1,\bar{p}=\infty$) by the case $\sigma \in L^\infty([0,T]\times\rd;\RR^{d\times d}) \cap L^2([0,T];W^{1,2}_{\mathrm{loc}}(\rd;\RR^{d\times d_1}))$, $b\in L^\infty([0,T]\times\rd)$, where $D_x b$ is a finite sum of certain singular integrals. We would like to remark that the class of such drift coefficients $b$ does not include $L^1([0,T];W^{1,1}_{\mathrm{loc}}(\rd;\rd))$, see Remark \ref{SDE.PU.remark.SIcounterexample}.
\end{remark}
This result confirms Trevisan's conjecture that the uniqueness result \cite[Theorem 1.1]{roeckner2010weakuniqueness} for \eqref{FPKE} is true for not necessarily bounded coefficients using Trevisan's superposition principle (cf. \cite[p.~131]{trevisan2014thesis}).

\section{Pathwise uniqueness for \eqref{SDE} with local Sobolev coefficients}\label{section.SDE.pathwiseUniqueness}
In this section, we prove a new pathwise uniqueness result for \eqref{SDE} with locally Sobolev regular coefficients, among probabilistically weak solutions  whose time-marginal law densities satisfy a certain Lebesgue integrability condition.
As already mentioned in the introductory part of this chapter, most of the results on pathwise uniqueness in the literature are centred around spacial continuity and/or degeneracy assumptions on the diffusion coefficient $\sigma$, see e.g. \cite{yamada1971yamada2,zvonkin74transformation,veretennikov1981strong,krylov05strong,fang2005nonlipschitz,zhang2011singular,ling2019sdes}.
In fact, the conditions to obtain probabilistically strong solutions to \eqref{SDE} in the previous sources reflect the conditions needed to prove pathwise uniqueness, since these strong existence results build upon the classical Yamada--Watanabe theorem and the conditions for pathwise uniqueness are stronger than those for weak existence in these sources.
In the case of spacial discontinuity of (time-dependent) $b,\sigma$ and possible degeneracy of $\sigma$, \cite{roeckner2010weakuniqueness}, \cite{luo2014uniq}, and \cite{champagnat2018strong} seem to provide one of the best multidimensional restricted pathwise uniqueness results up-to-date.
A careful observation of the techniques of the proof in \cite[Theorem 2.13 (ii)]{champagnat2018strong} and \cite[Theorem 1.1]{roeckner2010weakuniqueness} shows that we can localise the regularity conditions on $b$ and $\sigma$ in \cite[Theorem 2.13 (ii)]{champagnat2018strong}, resulting in the more general result presented in Theorem \ref{SDE.PU.theorem.mainresult.abstract}.
Similar to \cite[Theorem 2.13 (ii)]{champagnat2018strong}, Theorem \ref{SDE.PU.theorem.mainresult.abstract} incorporates an abstract condition on the Schwartz-distributional derivative of the drift coefficient $b$, which is why we present this result in the more feasible but special form of a local Sobolev regularity condition.

The following theorem is the main result of this section.
\begin{theorem}[pathwise uniqueness]\label{SDE.PU.theorem.mainresult}
	\sloppy Let $p\in [1,\infty), p' \in (1,\infty]$, such that
	$\nicefrac{1}{p}+\nicefrac{1}{p'}=1.$
Let $b \in L^p([0,T];W^{1,p}_{\mathrm{loc}}(\rd;\RR^{d}))$, $\sigma \in L^{2p}([0,T];W^{1,2p}_{\mathrm{loc}}(\rd;\RR^{d\times d_1}))$.

	Let $\nu \in \PPPP(\rd)$. Assume $(X,W),(Y,W)$ are two $P^{p', \mathrm{loc}}_\nu$-weak solutions to \eqref{SDE} on the same stochastic basis $(\Omega,\FF, \PP; (\FF_t)_{t\in [0,T]})$ with respect to the same Brownian motion $(W(t))_{t\in [0,T]}$ such that $X(0)=Y(0)$ $\PP$-a.e.
	Then
		$\sup_{t\in [0,T]}|X_t-Y_t| = 0$ $\PP$-a.s.
\end{theorem} 
\begin{remark}
\begin{enumerate}[label=(\roman*)]
\item In the case $p \in (1,\infty)$, Theorem \ref{SDE.PU.theorem.mainresult} 
is covered by the pathwise uniqueness result in the proof of \cite[Theorem 1.1]{roeckner2010weakuniqueness}.
\item To the best of the authors' knowledge, the case $p=1$ is not covered by the literature so far. In particular, the result is not proved in \cite[Proof of Theorem 1.1]{roeckner2010weakuniqueness}, see Remark \ref{SDE.PU.remark.OSLcounterexample}.
Furthermore, the case is also not covered by \cite[Proof of Theorem 1.2]{luo2014uniq}, 
 see Remark \ref{SDE.PU.remark.SIcounterexample}. 
\end{enumerate}
\end{remark}

\subsection*{A restricted pathwise uniqueness results from \cite{roeckner2010weakuniqueness} and a comparison with Theorem \ref{SDE.PU.theorem.mainresult}}
The following restricted pathwise uniqueness result is extracted from the proof of \cite[Theorem 1.1]{roeckner2010weakuniqueness}. For the notation we refer to Sections \ref{section.mainResults.strongSolutionsSDE} and \ref{section.mainResults.uniquenessFPKE}.
\begin{theorem}[monotonicity conditions]
 \label{SDE.PU.theorem.rz10}  Let $p, p'\in [1,\infty]$, such that
	$\nicefrac{1}{p}+\nicefrac{1}{p'}=1.$
Let $b \in L^p_{\mathrm{loc}}([0,T]\times\rd;\rd)$ and $\sigma \in L^{2p}_{\mathrm{loc}}([0,T]\times \rd;\RR^{d\times d_1})$ such that \ref{SDE.PU.condition.monotonicity} holds.
	Let $\nu \in \PPPP(\rd)$ and assume that $(X,W),(Y,W)$ are two $P^{p', \mathrm{loc}}_\nu$-weak solutions to \eqref{SDE} on a common stochastic basis $(\Omega,\FF,\PP; (\FF_t)_{t\in [0,T]})$ with respect to the same Brownian motion $(W(t))_{t\in [0,T]}$ and $X(0)=Y(0)$ $\PP$-a.s.
		Then $\sup_{t\in [0,T]}|X(t)-Y(t)|=0$.
\end{theorem}
\begin{proof}
	The proof is contained in the proof of \cite[Theorem 1.1]{roeckner2010weakuniqueness} for the case $d=d_1$ and $b \in L^p_{\mathrm{loc}}([0,T]\times\rd;\rd)$ and $\sigma \in L^{2p}_{\mathrm{loc}}([0,T]\times \rd;\RR^{d\times d_1})$. 
	However, the same lines of proof are valid for $d\neq d_1$ and the weaker conditions on $b$ and $\sigma$ stated above.
	Furthermore, in the proof of \cite[Theorem 1.1]{roeckner2010weakuniqueness} the authors considered weak solutions with time-marginal law densities in $L^{p'}([0,T];L^{p'}_{\mathrm{loc}}(\rd))$. However, considering weak solutions with time-marginal law densities  in $L^{p'}_{\mathrm{loc}}([0,T]\times\rd)$, as we do, does not change the proof as well.
	\end{proof}
As already indicated in the introduction to this section, Theorem \ref{SDE.PU.theorem.rz10} covers Theorem \ref{SDE.PU.theorem.mainresult} in the case $p\in (1,\infty)$. This is due to the boundedness of the (local) Hardy--Littlewoood maximal operator in $L^p$ for such $p$ and illustrated by the following lemma. This phenomenon has also been remarked in \cite[Remark 1.2]{roeckner2010weakuniqueness} before. For completeness, we provide the proof.

First we will recall the definition of the (local) Hardy--Littlewood maximal function and an important Lipschitz type estimate.
\begin{definition}
	Let $R\in [0,\infty]$. For $\mu \in \mathcal{M}_{\mathrm{loc}}(\rd;\RR^{d\times d_1})$, the (local) Hardy--Littlewood maximal function is defined as
\begin{align*}\M_{R}|\mu|(x):= \sup_{0< r\leq R}\frac{1}{(\mathrm{d}x)(B_r(0))}\int_{B_r(x)}|\mu|(dy), \text{ for all $x\in \rd$ }.
 \end{align*}
 In the case $R=\infty$, we will write $\M\equiv \M_\infty$.
\end{definition}
\begin{lemma}[{\cite[Lemma A.3]{crippa2008estimates}}]
	There exists a constant $C_d >0$ depending only on the dimension $d$ such that for all $f\in W^{1,1}_{\mathrm{loc}}(\rd;\rd)$ there exists $N \in \BBBB(\rd)$ with $(\mathrm{d}x)(N)=0$ and
\begin{align}\label{CrippaDeLellisestimate:ineq}
	|f(x)-f(y)| \leq C_d \left( \mathrm{M}_R|Df|(x) + \mathrm{M}_R|Df|(y) \right) |x-y| \text{ for all } x, y \in N^\complement \text{ such that } |x-y|\leq R.
\end{align}
\end{lemma} 
\begin{lemma}\label{SDE.PU.lemma.SobolevCoeffSatOSL}
	Let $p\in [1,\infty]$.
	Let
	\begin{align*}
		b \in L^p([0,T];W_{\mathrm{loc}}^{1,p}(\rd;\rd)), \ \ \sigma \in L^{2p}([0,T];W_{\mathrm{loc}}^{1,{2p}}(\rd;\RR^{d\times d_1})),
	\end{align*}
if $1<p\leq \infty$, and
\begin{align*}
	b \in L^1([0,T];W_{\mathrm{loc}}^{1,1+\varepsilon}(\rd;\rd)), \ \ \sigma \in L^{2}([0,T];W_{\mathrm{loc}}^{1,{2}}(\rd;\RR^{d\times d_1})),
\end{align*}
for some $\varepsilon >0$ in the case $p=1$.
Then \ref{SDE.PU.condition.monotonicity} holds.
\end{lemma}
\begin{proof}
For $(t,x) \in [0,T]\times\rd$ we define
\begin{align*}
	f_R(t,x):= 2C_d \M_{2R}|D b_t|(x) + C_d^2(\M_{2R}|D\sigma_t|(x))^2.
\end{align*}
Due to \eqref{CrippaDeLellisestimate:ineq}, $b$ and $\sigma$ satisfy \eqref{SDE.PU.condition.oneSidedLipschitz:eq} with such $f_R$.
 Furthermore, using the boundedness of the operator $\mathrm{M}_{2R}$ between local $L^p$-spaces (see, e.g. \cite[Lemma A.2]{crippa2008estimates}),
there exist constants $C_{d,p},C_{d,2p}>0$ just depending on $d$ and $p$ such that in the case $1<p < \infty$
 \begin{align*}
 	&\norm[{L^p([0,T]\times B_R(0))}]{f_R}\\
 	&\ \ \ \ \ \ \ \ \ \leq 2C_d\norm[{L^{p}([0,T]\times B_R(0))}]{\M_{2R}|D_x b|}^{} 
 	+ C_d^2\norm[{L^{2p}([0,T]\times B_R(0))}]{\M_{2R}|D_x \sigma|}^{2} \\
 	&\ \ \ \ \ \ \ \ \ \leq 2C_dC_{d,p}^\frac{1}{p}\norm[{L^{p}([0,T]\times B_{3R}(0))}]{|D_x b|}^{} 
 	+C_d^2C_{d,2p}^\frac{1}{p} \norm[{L^{2p}([0,T]\times B_{3R}(0))}]{|D_x \sigma|}^{2}<\infty.
 \end{align*}
 The same lines hold true in the case $p=\infty$ for $C_{d,\infty}:=1$.
 In the case $p=1$ we estimate similarly with the help of Jensen's inequality
 \begin{align*}
 	&\norm[{L^1([0,T]\times B_R(0))}]{f_R} \\
 	&\ \ \ \ \ \leq 2C_d|B_R|^{\frac{\varepsilon}{1+\varepsilon}}\norm[{L^{1}([0,T];L^{1+\varepsilon}(B_R(0)))}]{\M_{2R}|D_xb|}^{} 
 	+ C_d^2\norm[{L^{2}([0,T]\times B_R(0))}]{\M_{2R}|D_x \sigma|}^{2} \\
 	&\ \ \ \ \ \leq C_dC_{d,1+\varepsilon}^{\frac{1}{1+\varepsilon}}|B_R|^{\frac{\varepsilon}{1+\varepsilon}}\norm[{L^{1}([0,T];L^{1+\varepsilon}(B_{3R}(0)))}]{|D_x b|}^{} 
 	+C_d^2C_{d,2} \norm[{L^{2}([0,T]\times B_{3R}(0))}]{|D_x \sigma|}^{2}<\infty.
\end{align*}
\end{proof}
Now one can ask if (at least) every bounded $b \in L^1([0,T];W^{1,1}(\rd))$ satisfies  \ref{SDE.PU.condition.monotonicity} for $p=1$.
The answer is negative as the following remark shows.
\begin{remark}[bounded drifts in ${L^1([0,T];W^{1,1}(\RR))}$ do not generally satisfy (MC$_1^{b,\sigma}$), cf. \cite{hajlasz1996geometric}] \label{SDE.PU.remark.OSLcounterexample}
	We set the time-homogeneous drift coefficient $b$ as $b(x):= \frac{-x}{|x|\ln|x|}$ for $x\in \left(-\frac{1}{2}, \frac{1}{2}\right)\backslash\{0\}$, $b(0):=0$, and on $\RR\backslash \left(-\frac{1}{2}, \frac{1}{2}\right)$ in such a way that $b \in L^\infty(\RR)\cap W^{1,1}(\RR))$.
	Assume there exists $f \in L^1\left(\left(-\frac{1}{2}, \frac{1}{2}\right)\right)$ such that \eqref{SDE.PU.condition.oneSidedLipschitz:eq} holds. Since $b$ is nondecreasing on $\left(-\frac{1}{2}, \frac{1}{2}\right)$, it needs to satisfy
	\begin{align}\label{SDE.PU.remark.OSLcounterexample:1}
		2|b(x)-b(y)| \leq (f_{}(x) + f_{}(y))|x-y|
	\end{align}
	for almost every $x,y \in (-\frac{1}{2}, \frac{1}{2})$. Note that this inequality implies that $f \geq 0$ a.e.
	Now the same argument as in \cite[Example, p. 7]{hajlasz1996geometric} shows that \eqref{SDE.PU.remark.OSLcounterexample:1} cannot hold for such an $f$.
	From here, it is straightforward to see that $b$ does not satisfy \ref{SDE.PU.condition.monotonicity} for $p=1$.
\end{remark}
\subsection*{A comparison of the restricted pathwise uniqueness result from \cite{luo2014uniq} with Theorem \ref{SDE.PU.theorem.mainresult}}
Carefully combining the techniques from \cite{roeckner2010weakuniqueness, bouchut2013lagrangian}, Luo \cite[Proof of Theorem 1.2]{luo2014uniq} proved a restricted pathwise uniqueness result for \eqref{SDE} with bounded $b$ and $\sigma$ among probabilistic weak solutions with bounded time-marginal law densities. Therefore, he needed to assume that $\sigma \in L^2([0,T];W^{1,2}_{\mathrm{loc}}(\rd))$, $b\in L^1_{\mathrm{loc}}([0,T]\times\rd)$, and that the latter's Schwartz-distributional spacial derivative can be written as a finite sum of certain singular integrals.
This class of drift coefficients includes $L^1([0,T];W^{1,1}(\rd))$. In any case, it is assumed that $D_x b(t,\cdot) \in L^1_{\text{w}}(\rd;\RR^{d\times d})$ for almost every $t\in [0,T]$, where $L^1_{\text{w}}(\rd;\RR^{d\times d})$ consists of all Borel measurable functions $f:\rd 
\to \RR^{d\times d}$ such that 
\begin{align*}
	||f||_{L^1_{\text{w}}(\rd)}:= \sup_{\lambda >0} \lambda \cdot (\mathrm{d}x)(\{x \in \rd : |f|> \lambda \})<\infty.
\end{align*}
The subsequent remark shows that Theorem \ref{SDE.PU.theorem.mainresult} is not covered by the restricted pathwise uniqueness result proved in  \cite[Proof of Theorem 1.2]{luo2014uniq}. 
\begin{remark}[$W^{1,1}_{\mathrm{loc}}(\RR)\not\subseteq \{f\in L^1_{\mathrm{loc}}(\RR): f' \in L^1_{\text{w}}(\RR)\}$]\label{SDE.PU.remark.SIcounterexample}
We define $f(x):= \sum_{i \in \mathbb{Z}} \mathbbm{1}_{[i,i+2)}(x)|x-(i+1)|^{\frac{1}{2}}, x\in \rd$.
It standard to see that $f \in L^1_{\mathrm{loc}}(\RR)$ and that $f$ is weakly differentiable with (a.e. determined) weak derivative
\begin{align*}
	f'(x)= \sum_{i \in \mathbb{Z}} \mathbbm{1}_{[i,i+2)}(x) |x-(i+1)|^{-\frac{3}{2}}(x-(i+1)),\ \ x\in \RR.
\end{align*}
Since $(\mathrm{d}x)(\{x \in \RR : |f'|> \lambda \})=+\infty$ for all $\lambda>0$, $||f'||_{L^1_{\text{w}}(\RR)}=+\infty$. Hence, $W^{1,1}_{\mathrm{loc}}(\RR)\not\subseteq \{f\in L^1_{\mathrm{loc}}(\RR): f' \in L^1_{\mathrm{w}}(\RR)\}$.
\end{remark}
\subsection*{Proof of Theorem \ref{SDE.PU.theorem.mainresult}.}
In order to prove Theorem \ref{SDE.PU.theorem.mainresult}, we need some preparation. 
In \cite{champagnat2018strong}, the authors obtained a Lipschitz type estimate similar to \eqref{CrippaDeLellisestimate:ineq} replacing the usual Hardy--Littlewood maximal operator $\M\equiv\M_\infty$ by operators $\M_L, L>1$, resulting in estimate \eqref{SDE.PU.lemma.MRL.estimate:ineq} in the case $R=\infty$. The reason for this is that $\M_L(\mu), L>1,$ can be controlled nicely if $\mu\in L^1$, in contrast to $\M(\mu)$ (see \cite[p.~7]{stein1970singular}).
Employing the estimate \eqref{SDE.PU.lemma.MRL.estimate:ineq} in the case $R = \infty$, the authors of \cite{champagnat2018strong} proved a pathwise uniqueness result among weak solutions to \eqref{SDE} with bounded time-marginal law densities under general conditions on $b$ and $\sigma$, which includes the case that the coefficients are bounded and $b\in L^1([0,T];W^{1,1}(\rd))$ and $\sigma \in L^2([0,T];W^{1,2}(\rd))$.

In the following, we introduce a local version of the functional $\M_L$ on $\mathcal{M}_{\mathrm{loc}}(\rd;\RR^{d\times d_1})$. These functionals will be called $\M_L^R$, $R>0$, and are essential in the proof of Theorem \ref{SDE.PU.theorem.mainresult.abstract}, from which we will conclude Theorem \ref{SDE.PU.theorem.mainresult}.
Here, 'local' means that $\M_L^R(\mu)$ depends only on $\left.\mu\right|_{\BBBB(B_R(0))}$, where $B_R(0)$ denotes the usual open ball in $\rd$, $R>0$.
We have the following definition.

\begin{definition}\label{SDE.PU.definition.MRL} Let $n,m\in \LN,\ L>1,\ R>0$. Let $\mu\in \mathcal{M}_{b}(B_R(0);\RR^{n\times m})$. For $x\in \rd$, we set
	\begin{align}
		\M^R_L(\mu) (x):= \sqrt{\ln(L)} + \int_{B_R(0)} \frac{|\mu_a(z)|\mathbbm{1}_{|\mu_a| \geq \sqrt{\ln(L)}}(z)\mathrm{d}z+|\mu_s|(\mathrm{d}z)}{(L\inv + |x-z|)|x-z|^{d-1}}\ \ (\in [0,\infty]),
	\end{align}
	where $\mu_a$ denotes the density of the absolutely continuous part and $\mu_s$ the singular part of the Radon measure $\mu$ with respect to Lebesgue measure according to Lebesgue's decomposition theorem (cf. \cite[Theorem 2.22]{ambrosio2000BV}).
\end{definition}
The following lemma is a variant of \cite[Lemma 3.2]{champagnat2018strong}, in which they proved the following lemma for $R=\infty$ and a.e. $x,y\in \rd$.
Having a closer look at the proof of \cite[Lemma 3.2]{champagnat2018strong}, and restricting our choice of $x, y$ to a ball $B_R(0)$, it shows that the following is true.
\begin{lemma}\label{SDE.PU.lemma.MRL.estimate}
Let $n\in\LN$, $L>1$ and assume that $f \in BV_{\mathrm{loc}}(\RR^d;\RR^n)$.
Then there exists a constant $C_d>0$ depending only on the dimension $d$ such that for almost every $x,y \in B_R(0)$
\begin{align}\label{SDE.PU.lemma.MRL.estimate:ineq}
|f(x)-f(y)| \leq C_d \big(h_R^{L\inv}(x)+h_R^{L\inv}(y)\big) (|x-y| + L\inv),
\end{align}
where $h_R^{L\inv} := |f| + \M^{R+2}_L (D f)$.
\end{lemma}
\begin{proof}
	In fact, the proof of \cite[Lemma 3.2]{champagnat2018strong} is the same here.
\end{proof}
Before we give the abstract main result, we need the following definition.
\begin{definition}\label{definition.superlinearity}
Let $\phi:(1,\infty)\to[0,\infty)$. $\phi$ is said to be \textit{super-linear of modest growth}, if
\begin{itemize}
	\item the map $(1,\infty)\ni r\mapsto \frac{\phi(r)}{r}$ is non-decreasing,
	\item $\frac{\phi(r)}{r}\to \infty$, for $r\to \infty$,
	\item $\frac{\phi(r)}{r\ln(r)}\to 0, \text{ for } r\to \infty$.
\end{itemize}
\end{definition}

\begin{theorem}\label{SDE.PU.theorem.mainresult.abstract}
	Let $p \in [1,\infty), p' \in (1,\infty]$ such that
	$\nicefrac{1}{p}+\nicefrac{1}{p'}=1$ and let $b \in L^p_{\mathrm{loc}}([0,T]\times\rd;\rd)$, $\sigma \in L^{2p}_{\mathrm{loc}}([0,T]\times\rd;\RR^{d\times d_1})$ such that $b(t,\cdot) \in BV_{\mathrm{loc}}(\rd;\rd)$ for a.e. $t \in [0,T]$.
	Let $\nu \in \PPPP(\rd)$. Assume that $(X,W),(Y,W)$ are two $P^{p',\mathrm{loc}}_\nu$-weak solutions to \eqref{SDE} on a common stochastic basis $(\Omega,\FF,\PP; (\FF_t)_{t\in [0,T]})$ with respect to the same Brownian motion $(W(t))_{t\in [0,T]}$, and $X(0)=Y(0)$ $\PP$-a.s.

	Assume that for every radius $R>0$, there exists a function $f_R \in L^p([0,T]\times B_R(0))$, a super-linear function $\phi_R$ of modest growth, $\varepsilon_0=\varepsilon_0(R)\in (0,1)$, and a constant $C_{T,R}>0$, which only depends on $T$ and $R$, such that for almost all $(t,x,y) \in [0,T]\times B_R(0)\times B_R(0)$
	\begin{align}\label{SDE.PU.oneSidedLipschitz}
		|\sigma(t,x)-\sigma(t,y)|^2 \leq (f_R(t,x)+f_R(t,y)) |x-y|^2,
	\end{align}
	and, for all $0<\varepsilon<\varepsilon_0$
	and all $\gamma \in \rd$,
	\begin{align}\label{SDE.PU.newMaximalFunction.condition}
		\int_0^T\int_{B_R(0)} \M^{R+2}_{\varepsilon\inv}(D_x b(t,\cdot))(x-\gamma) (u_X(t,x)+u_Y(t,x))\mathrm{d}x\mathrm{d}t \leq C_{T,R}\frac{|\ln(\varepsilon)|}{\varepsilon\phi_R(\varepsilon\inv) },
	\end{align}
	where we set $u_X(t) := \nicefrac{\mathrm{d}\law{X(t)}}{\mathrm{d}x}, u_Y(t) := \nicefrac{\mathrm{d}\law{Y(t)}}{\mathrm{d}x}$, for all $t\in [0,T]$.
Then $\sup_{t\in [0,T]}|X(t)-Y(t)|=0$ $\PP$-a.s.
\end{theorem}
The following proof is a modification of the proof of \cite[Theorem 2.13 (ii)]{champagnat2018strong} 
and it employs the stopping time techniques used in the proof of \cite[Theorem 1.1]{roeckner2010weakuniqueness}.
\begin{proof}[Proof of Theorem \ref{SDE.PU.theorem.mainresult.abstract}]
For the convenience of the reader, we will write $b_t(x):=b(t,x)$, $\sigma_t(x):=\sigma(t,x)$, and $f_{R,t}(x):=f_R(t,x)$, for $(t,x) \in [0,T]\times\rd, R>0$.

For each $\varepsilon > 0$, let $L_\varepsilon \in C^\infty(\RR^d)$ with $0\leq L_\varepsilon\leq 1$ such that
\begin{align}\label{PU_cutoffL}
L_\varepsilon = 1 \text{ in } B_\varepsilon(0)^\complement, L_\varepsilon = 0 \text{ in } B_{\nicefrac{\varepsilon}{2}}(0) \text{ and } \varepsilon\norm{\nabla L_\varepsilon}_{L^\infty(\rd)} + \varepsilon^2\norm{D^2L_\varepsilon}_{L^\infty(\rd)} \leq C,
\end{align}
for some constant $C>0$ independent of $\varepsilon$. 
Furthermore, let $\varphi \in C_c^\infty(\rd), \varphi \geq 0$ with $\int_\rd \varphi\   \mathrm{d}x =1$ and $\text{supp}\ \varphi \subseteq B_1(0)$. We define the usual Dirac sequence with respect to $\varphi$ via
$\varphi_\delta := \delta^{-d} \varphi(\nicefrac{\cdot}{\delta})$, $\delta \in (0,1)$.

Let $R>0$ and let us fix an arbitrary $\varepsilon<\varepsilon_0$.
Furthermore, let $${\tau_R(\omega):= \inf\{t \in [0,T] : \max\{|X_t(\omega)|,|Y_t(\omega)|\} \geq R\}},\ \ \omega \in \Omega,$$ where we use the convention $\inf \emptyset = \infty$. Note that $\tau_R$ is an $(\FF_t)$-stopping time.
By It\^o's formula, we have
\begin{align}\label{SDE.PU.theorem.mainresult.abstract:ito}
\E&[L_\varepsilon(X_{t\wedge\tau_R}-Y_{t\wedge\tau_R})]\nonumber\\
&= \frac{1}{2} \E\int_0^{t\wedge\tau_R} \mathrm{tr}\{D^2 L_\varepsilon(X_s-Y_s)(\sigma_s(X_s)-\sigma_s(Y_s))(\sigma_s(X_s)-\sigma_s(Y_s))^*\}\ \mathrm{d}s \nonumber \\ 
 &\ \ \ + \E\int_0^{t\wedge\tau_R} \nabla L_\varepsilon(X_s-Y_s) \cdot (b_s(X_s)-b_s(Y_s))\ \mathrm{d}s \nonumber \\ 
 &\ \ \ + \E\int_0^{t\wedge\tau_R} (\nabla L_\varepsilon(X_s-Y_s))^* (\sigma_s(X_s)-\sigma_s(Y_s))\ \mathrm{d}W_s,
\end{align}
where tr$\{\cdot\}$ denotes the usual trace of a square matrix.
Clearly, the last summand in \eqref{SDE.PU.theorem.mainresult.abstract:ito} vanishes.
Let $\delta\in (0,1)$. By the estimates on the derivatives of $L_\varepsilon$ (see \eqref{PU_cutoffL}), we know that
\begin{align}\label{SDE.PU.theorem.mainresult.abstract:ito2}
&\E[L_\varepsilon(X_{t\wedge\tau_R}-Y_{t\wedge\tau_R})]  
\leq C \E\int_0^{t\wedge\tau_R} \mathbbm{1}_{\frac{\varepsilon}{2} \leq |X_s-Y_s|\leq\varepsilon} \bigg( \frac{|\sigma^\delta_s(X_s)-\sigma^\delta_s(Y_s)|^2}{\varepsilon^2} \nonumber \\
&+ \frac{|b^\delta_s(X_s)-b^\delta_s(Y_s)|}{\varepsilon}+  \frac{|\sigma^\delta_s(X_s)-\sigma_s(X_s)|^2}{\varepsilon^2} + \frac{|\sigma^\delta_s(Y_s)-\sigma_s(Y_s)|^2}{\varepsilon^2} \nonumber\\
&+\frac{|b^\delta_s(X_s)-b_s(X_s)|}{\varepsilon}+ \frac{|b^\delta_s(Y_s)-b_s(Y_s)|}{\varepsilon}\bigg) \ \mathrm{d}s,
\end{align}
where $\sigma^\delta_t:=\sigma_t\ast\varphi_\delta$, $b^\delta_t:=b_t\ast\varphi_\delta$, $t\in [0,T]$. Here, the convolution is meant componentwise in the functions' $x$-variable. 

In the following, we will further estimate the right-hand side of  \eqref{SDE.PU.theorem.mainresult.abstract:ito2} by employing Lemma \ref{SDE.PU.lemma.MRL.estimate}, and show that there is a sequence of natural numbers $(n_i)_{i\in\LN}$, such that
\begin{align}\label{SDE.PU.theorem.mainresult.abstract:1}
\sup_{t\in[0,T]}\E[L_{2^{-n_i}}(X_{t\wedge\tau_R}-Y_{t\wedge\tau_R})] \to 0, \text{ as } i\to+\infty.
\end{align}
From there, we will then argue that $X=Y$ $\PP$-a.s.

Let us set  ${h}_{R}^\varepsilon(t,x):={h}_{R,t}^\varepsilon(x) := |b_t(x)| + M^{R+2}_{\varepsilon\inv}(D b_t)(x), (t,x)\in [0,T]\times\rd$.
As above, let us set $({h}_{R,t}^\varepsilon)^\delta := {h}_{R,t}^\varepsilon\ast \varphi_\delta$, as well as $(f_{R,t})^\delta := f_{R,t}\ast \varphi_\delta$.
Using Lemma \ref{SDE.PU.lemma.MRL.estimate} and \eqref{SDE.PU.oneSidedLipschitz}, respectively, it is an easy direct calculation to see that, for almost every $t\in [0,T]$, and \textit{all} $x,y\in B_R(0)$,
\begin{align*}
	|\sigma^\delta_t(x)-\sigma^\delta_t(y)|^2 \leq ((f_{R+1,t})^\delta(x) + (f_{R+1,t})^\delta(y))|x-y|^2,
\end{align*}
and
\begin{align*}
	|b^\delta_t(x)-b^\delta_t(y)| \leq C_d \big(({h}_{R+1,t}^\varepsilon)^\delta(x)+({h}_{R+1,t}^\varepsilon)^\delta(y)\big) (|x-y| + \varepsilon).
\end{align*}
Here, $C_d>0$ is the same constant as in Lemma \ref{SDE.PU.lemma.MRL.estimate}.
Applying these inequalities, we further estimate for $\tilde{C}_d:= \max\{C,2CC_d\}$
\begin{align}\label{SDE.PU.theorem.mainresult.abstract:2}
\E[L_\varepsilon &(X_{t\wedge\tau_R}-Y_{t\wedge\tau_R})] \\ 
\leq\ & \tilde{C}_d \E\int_0^{T\wedge\tau_R} \mathbbm{1}_{\frac{\varepsilon}{2} \leq |X_s-Y_s|\leq\varepsilon}\ \left((f_{R+1,s})^\delta(X_s) + (f_{R+1,s})^\delta(Y_s)\right) \ \mathrm{d}s \nonumber\\
&+ \tilde{C}_d \E\int_0^{T\wedge\tau_R} \mathbbm{1}_{\frac{\varepsilon}{2} \leq |X_s-Y_s|\leq\varepsilon}\ 
\left(({h}_{R+1,s}^\varepsilon)^{\delta}(X_s) + ({h}_{R+1,s}^\varepsilon)^{\delta}(Y_s)\right)\ \mathrm{d}s \nonumber\\
&+ \tilde{C}_d \E \int_0^{T\wedge\tau_R} \mathbbm{1}_{\frac{\varepsilon}{2} \leq |X_s-Y_s|\leq\varepsilon}\ \left(\frac{|b^\delta_s(X_s)-b_s(X_s)|}{\varepsilon} + \frac{|b^\delta_s(Y_s)-b_s(Y_s)|}{\varepsilon}\right)\ \mathrm{d}s \nonumber \\
&+ \tilde{C}_d \E \int_0^{T\wedge\tau_R} \mathbbm{1}_{\frac{\varepsilon}{2} \leq |X_s-Y_s|\leq\varepsilon}\ \left(\frac{|\sigma^\delta_s(X_s)-\sigma_s(X_s)|^2}{\varepsilon^2} + \frac{|\sigma^\delta_s(Y_s)-\sigma_s(Y_s)|^2}{\varepsilon^2}\right)\ \mathrm{d}s \nonumber \\
=: &\ \tilde{C}_d\left(I_1^{\delta,\varepsilon} + I_2^{\delta,\varepsilon} + I_3^{\delta,\varepsilon} + I_4^{\delta,\varepsilon}\right).\nonumber
\end{align}
Let us first regard the summands $ I_3^{\delta,\varepsilon}$ and $I_4^{\delta,\varepsilon}$.
Note that 
$b^\delta \to b$ in $L^p([0,T]\times B_R(0);\rd)$ and $\sigma^\delta \to \sigma$ in $L^{2p}([0,T]\times B_R(0);\RR^{d\times d_1})$, whenever $\delta\to 0$.
Hence, by H\"older's inequality
\begin{align*}
	I_3^{\delta,\varepsilon} + I_4^{\delta,\varepsilon}
	&\leq\varepsilon^{-1}\norm[{L^p([0,T]\times B_R(0);\rd)}]{b^\delta - b} \left(\norm[{L^{p'}([0,T]\times B_R(0))}]{u_X}+\norm[{L^{p'}([0,T]\times B_R(0))}]{u_Y}\right) \\
	&+ \varepsilon^{-2}\norm[{L^{2p}([0,T]\times B_R(0);\RR^{d\times d_1})}]{\sigma^\delta - \sigma} \left(\norm[{L^{p'}([0,T]\times B_R(0))}]{u_X}+\norm[{L^{p'}([0,T]\times B_R(0))}]{u_Y}\right) \\
	&\to 0, \text{ as $\delta \to 0$.}
\end{align*}
It will be fairly easy to argue that $\lim_{\varepsilon\to 0}\liminf_{\delta\to 0} I_1^{\delta,\varepsilon}=0$. 
The summand $I_2^{\delta,\varepsilon}$, however, needs to be treated more carefully, since ${h}_\varepsilon^R$ depends on the parameter $\varepsilon$.
Therefore, we will argue that ${I}_2^{\delta,\varepsilon}$ can be replaced by $\bar{I}_2^{\delta,\varepsilon}$, where the latter incorporates a discretised modification $\bar{h}_R^\varepsilon$ of ${h}_R^\varepsilon$ in the variable $\varepsilon$. Afterwards, we show that for a suitable subsequence $(n_i)_{i\in \LN}$, $\lim_{i\to \infty}\liminf_{\delta\to 0} \bar{I}^{\delta,2^{-n_i}}_2 = 0$. 

We partition the open unit interval $(0,1) \subseteq \RR$ into pairwise disjoint half-open intervals $I_i$, $i\in \LN$, as follows. 
Let $I_0 := [\nicefrac{1}{2},1)$ and  define for each $i \in \LN$, $I_i := [a_i, c_i)$ with $c_{i} = a_{i-1}$ and $a_i = c_i^2$, starting with $a_0 := \nicefrac{1}{2}$. 
In particular, for all $i \in \LN$ we have
\begin{align*}
\ln(a_i\inv) = 2\ln(c_i\inv).
\end{align*}
Recall that by the isotonicity of $r\mapsto \nicefrac{\phi_R(r)}{r}$, we have
\begin{equation*}
\frac{\phi_R(c_i\inv)}{c_i\inv} \leq \frac{\phi_R(a_i\inv)}{a_i\inv}.
\end{equation*}
For each $i\in\LN$ and $\varepsilon \in I_i$, we set $\bar{h}^\varepsilon_{R,t} := {h}^{a_i}_{R,t}, t \in [0,T]$.
Then, by \eqref{SDE.PU.newMaximalFunction.condition},
we can find a constant $\bar{C}_{T,R}>0$ such that for $i \in \LN$ large enough that $\nicefrac{\phi_R(a_i\inv)a_i}{|\ln(a_i)|}\leq 1$ and $a_i \leq \varepsilon_0$
\begin{align}\label{SDE.PU.theorem.mainresult.abstract:3}
\int_0^T \int_{B_R(0)} (\bar{h}^\varepsilon_{R,t})^\delta(x)\ (u_X(t,x) + u_Y(t,x))\ \mathrm{d}x\mathrm{d}t
&\leq \bar{C}_{T,R} \frac{|\ln(a_i)|}{a_i \phi_R(a_i\inv)} \\
&\leq \bar{C}^{}_{T,R} \frac{|\ln(a_i)|}{c_i \phi_R(c_i\inv)}
= 2\bar{C}^{}_{T,R} \frac{|\ln(c_i)|}{c_i \phi_R(c_i\inv)}.\nonumber
\end{align}
Since for each $\varepsilon \in I_i$ we have $a_i \leq \varepsilon$, \eqref{SDE.PU.theorem.mainresult.abstract:2} can be similarly obtained in terms of $\bar{h}_{R,t}^\varepsilon$ instead of ${h}_{R,t}^\varepsilon$. To be more precise, we can estimate
\begin{align}\label{SDE.PU.theorem.mainresult.abstract:4}
\E[L_\varepsilon &(X_{t\wedge\tau_R}-Y_{t\wedge\tau_R})] \\ 
\leq\ & \tilde{C}_d \E\int_0^{T\wedge\tau_R} \mathbbm{1}_{\frac{\varepsilon}{2} \leq |X_s-Y_s|\leq\varepsilon}\ \left((f_{R+1,s})^\delta(X_s) + (f_{R+1,s})^\delta(Y_s)\right) \ \mathrm{d}s \nonumber\\
&+ \tilde{C}_d \E\int_0^{T\wedge\tau_R} \mathbbm{1}_{\frac{\varepsilon}{2} \leq |X_s-Y_s|\leq\varepsilon}\ 
\left((\bar{h}_{R+1,s}^\varepsilon)^{\delta}(X_s) + (\bar{h}_{R+1,s}^\varepsilon)^{\delta}(Y_s)\right)\ \mathrm{d}s \nonumber\\
&+ \tilde{C}_d \E \int_0^{T\wedge\tau_R} \mathbbm{1}_{\frac{\varepsilon}{2} \leq |X_s-Y_s|\leq\varepsilon}\ \left(\frac{|b^\delta_s(X_s)-b_s(X_s)|}{\varepsilon} + \frac{|b^\delta_s(Y_s)-b_s(Y_s)|}{\varepsilon}\right)\ \mathrm{d}s \nonumber \\
&+ \tilde{C}_d \E \int_0^{T\wedge\tau_R} \mathbbm{1}_{\frac{\varepsilon}{2} \leq |X_s-Y_s|\leq\varepsilon}\ \left(\frac{|\sigma^\delta_s(X_s)-\sigma_s(X_s)|^2}{\varepsilon^2} + \frac{|\sigma^\delta_s(Y_s)-\sigma_s(Y_s)|^2}{\varepsilon^2}\right)\ \mathrm{d}s \nonumber \\
&=: \tilde{C}_d (I_1^{\delta,\varepsilon} + \bar{I}_2^{\delta,\varepsilon} + I_3^{\delta,\varepsilon} + I_4^{\delta,\varepsilon}).\nonumber
\end{align}
Substituting $\varepsilon = 2^{-k}$, $k\in \LN$, in $I_1^{\delta,\varepsilon}$ we obtain by Fatou's lemma and H\"older's inequality
\begin{align}\label{SDE.PU.theorem.mainresult.abstract:5}
\sum_{k\in\LN} \liminf_{\delta \to 0}I_1^{\delta,2^{-k}} 
&\leq
\liminf_{\delta \to 0}\int_0^{T} \int_{B_R(0)} (f_{R+1,t})^\delta(x)\ (u_X(t,x) + u_Y(t,x))\ \mathrm{d}x\mathrm{d}t  \\
&\leq \norm[{L^{p}([0,T]\times B_{R+1}(0))}]{f_{R+1}}(\norm[{L^{p'}([0,T]\times B_R(0))}]{u_X}+\norm[{L^{p'}([0,T]\times B_R(0))}]{u_Y}),\notag
\end{align}
where the right-hand side is finite due to the integrability assumptions on $f_{R+1}$, $u_X$ and $u_Y$. It is immediate to see that this implies $\lim_{k\to \infty }\liminf_{\delta \to 0}I_1^{\delta,2^{-k}}=0$.
Regarding $\bar{I}_2^{\delta,\varepsilon}$, we now define $J_i := \{ k \in \LN : [2^{-(k+1)},2^{-k}) \subseteq I_i \}$ for each $i \in \LN$. By Remark \ref{SDE.PU.theorem.mainresult.abstract:remark} there exists a constant $C_J > 0$ independent of $i$, such that $|J_i| = C_J\inv |\ln(c_i)|$. By \eqref{SDE.PU.theorem.mainresult.abstract:3} and Fatou's lemma, we have for $i\in\LN$ large enough such that $\nicefrac{\phi(c_i\inv)c_i}{|\ln(c_i)|}\leq 1$ and $c_i \leq \varepsilon_0$
\begin{align*}
\frac{1}{|J_i|} \sum_{k\in J_i} \liminf_{\delta \to 0}\bar{I}_2^{\delta,2^{-k}}
&\leq \liminf_{\delta \to 0}\frac{1}{|J_i|} \bigg( \int_0^T \int_{B_R(0)} (\bar{h}_{R+1,s}^{a_i})^{\delta}(x)\ (u_X(s,x) + u_Y(s,x))\ \mathrm{d}x\mathrm{d}s \\
&\ \ \ \ + \int_0^T \int_{B_R(0)} (\bar{h}_{R+1,s}^{c_i})^{\delta}(x)\ (u_X(s,x) + u_Y(s,x))\ \mathrm{d}x\mathrm{d}s \bigg) \\
&\leq \frac{3\bar{C}_{T,R}C_J}{c_i\phi_R(c_i\inv)} \to 0 \text{, as } i \to \infty.
\end{align*}
Now for each $i\in\LN$ we set $n_i := \text{argmin}_{k \in J_i} (\liminf_{\delta\to 0} \bar{I}_2^{\delta,2^{-k}})$. It follows immediately that \\$\lim_{i\to \infty}\liminf_{\delta\to 0}\bar{I}_2^{\delta,2^{-n_i}} = 0$. 
Consequently, we have proved \eqref{SDE.PU.theorem.mainresult.abstract:1} for the sequence $(n_i)_{i\in\LN}$. 
Furthermore, by definition of the family $(L_\varepsilon)_{\varepsilon>0}$, we have
\begin{align}\label{SDE.PU.theorem.mainresult.abstract:6}
\E [L_{\varepsilon}(X_{t\wedge \tau_R}-Y_{t\wedge \tau_R})] \geq \PP(\{|X_{t\wedge \tau_R}-Y_{t\wedge \tau_R}| > \varepsilon\}).
\end{align}
Since $n_i \to \infty$, as $i\to\infty$, we have that $\{|X_{t\wedge \tau_R}-Y_{t\wedge \tau_R}| > 2^{-n_i}\} \uparrow_{i\in\LN} \{|X_{t\wedge \tau_R}-Y_{t\wedge \tau_R}| > 0\}$.
So, combining \eqref{SDE.PU.theorem.mainresult.abstract:1} with \eqref{SDE.PU.theorem.mainresult.abstract:6}, we then deduce (by the continuity of measures from below)
\begin{align*}
\PP(\{|X_{t\wedge \tau_R}-Y_{t\wedge \tau_R}| > 0\}) = 0\ \ \forall t\in [0,T].
\end{align*}
Since $X$ and $Y$ have $\PP$-a.s. continuous paths, we obtain $\tau_R \to T$ $\PP$-a.s., as  $R\to \infty$.
We therefore conclude that $\PP$-a.s.
\begin{align*}
	X_t = Y_t\ \ \ \forall t\in [0,T].
\end{align*}
This concludes the proof.
\end{proof}
\begin{remark}\label{SDE.PU.theorem.mainresult.abstract:remark}
Let us consider the partition $(I_i)_{i \in \LN}$ of the unit interval $(0,1)$ as above with $I_i=[a_i,c_i), i \in \LN_0$. Fix $i \in \LN$.
Let $k_0 \in \LN$ such that $2^{-k_0} = c_i$, which is equivalent to $k_0 = \log_2(c_i\inv)$.
 Then $|J_i|$ is given through the equation
\begin{align*}
2^{-k_0 - |J_i|} = a_i.
\end{align*}
An easy calculation consequently yields $|J_i| = \frac{|\ln(c_i)|}{\ln(2)}$.
\end{remark}
\begin{proposition}\label{SDE.PU.proposition.SobolevMLR}
Let $p\in [1,\infty), p'\in (1,\infty]$ such that
	$\nicefrac{1}{p}+\nicefrac{1}{p'}=1$.
Let $b \in L^1_{\mathrm{loc}}([0,T]\times\rd;\RR^n)$ such that $D_x b \in L^p_{\mathrm{loc}}([0,T]\times\rd;\RR^{n\times d})$ and $v \in L^{p'}_{\mathrm{loc}}([0,T]\times\rd)$.

	Then for every $R>0$ there exists a function $\phi_R:(1,\infty) \to [0,\infty)$ which is super-linear of modest growth, a constant $C_{T,R}>0$ such that for all $\varepsilon \in (0,1)$ small enough and all $\gamma \in \rd$
	\begin{align}\label{}
		\int_0^T\int_{B_R(0)} \M^{R}_{\varepsilon\inv}(D b_t)(x-\gamma) v_t(x)\ \mathrm{d}x\mathrm{d}t \leq C_{T,R}\frac{|\ln(\varepsilon)|}{\varepsilon\phi_R(\varepsilon\inv) }.
	\end{align}
\end{proposition}
\begin{proof}
	Let us fix $R>0$, $\gamma \in \rd$ and let $l_\gamma(x) := x-\gamma$, $x\in \rd$. We set $C_v^q:=\norm[{L^q([0,T]\times B_R(0))}]{v}, q\in [1,p']$.
	Without loss of generality, we assume $C_v^1\in (0,\infty)$. We have the following.
	\begin{align*}
		&\int_0^T \int_{B_R(0)} \M^R_{\varepsilon\inv}(D b_t)(l_\gamma(x)) v_t(x)\ \mathrm{d}x\mathrm{d}t\\
		&= C_v^1\sqrt{|\ln(\varepsilon)|} + \int_0^T\int_{B_R(0)} \int_{B_R(0)}  \frac{|D b_t(z)|\mathbbm{1}_{|D b_t| \geq \sqrt{|\ln(\varepsilon)|}}(z)}{(\varepsilon + |l_\gamma(x)-z|)|l_\gamma(x)-z|^{d-1}}\ \mathrm{d}z\ v_t(x)\ \mathrm{d}x\mathrm{d}t\\
		&= C_v^1\sqrt{|\ln(\varepsilon)|} + \int_0^T\int_{B_R(0)} \int_{B_R(0)}  \frac{|D b_t(z)|\mathbbm{1}_{|D b_t| \geq \sqrt{|\ln(\varepsilon)|}}(z)}{(\varepsilon + |l_\gamma(x)-z|)|l_\gamma(x)-z|^{d-1}}v_t(x)\ \mathrm{d}x\mathrm{d}z\mathrm{d}t\\
		&= C_v^1\sqrt{|\ln(\varepsilon)|}\\
		 &\ \ \ + \int_0^T\int \int \mathbbm{1}_{B_R(0)}(z-l_\gamma(x)) v_t(z-l_\gamma(x)) \mathbbm{1}_{B_R(0)}(z)\frac{|D b_t(z)|\mathbbm{1}_{|D b_t| \geq \sqrt{|\ln(\varepsilon)|}}(z)}{(\varepsilon + |x|)|x|^{d-1}}\ \mathrm{d}x \mathrm{d}z \mathrm{d}t \\
		&\leq C_v^1\sqrt{|\ln(\varepsilon)|} \\
		&\ \ \ + C_v^{p'}\norm[{L^{p}([0,T]\times B_R(0))}]{ |D_x b|\mathbbm{1}_{|D_x b| \geq \sqrt{|\ln(\varepsilon)|}}}\norm[{L^{1}(B_1(0))}]{\frac{1}{(\varepsilon + |\cdot|)|\cdot|^{d-1}}} \\
		&\ \ \ + C_v^1\norm[{L^{p}([0,T]\times B_R(0))}]{|D_x b|\mathbbm{1}_{|D_x b| \geq \sqrt{|\ln(\varepsilon)|}}}\norm[{L^{p'}( B_1(0)^\complement)}]{\frac{1}{(\varepsilon + |\cdot|)|\cdot|^{d-1}}},
	\end{align*}
	where we used Fubini's theorem in the third line, the general transformation rule for integrals with respect to the $x$-variable in the fourth line.
	For the inequality, we split the integral with respect to the $x$-variable into an integral over $B_1(0)$ and $B_1(0)^\complement$, respectively, and used Young's inequality for convolution in each resulting summand.
	Clearly, there exists a constant $C_{d}>0$ such that
	\begin{align*}
		\norm[L^{p'}(B_1(0)^\complement)]{\frac{1}{(\varepsilon + |\cdot|)|\cdot|^{d-1}}}\leq C_{d},
	\end{align*}
	and
	\begin{align*}
		\norm[{L^{1}(B_1(0))}]{\frac{1}{(\varepsilon + |\cdot|)|\cdot|^{d-1}}} = w_d \ln(1+\varepsilon\inv),
	\end{align*}
	where $w_d(>0)$ denotes the $(d-1)$-dimensional volume of the $(d-1)$-dimensional unit sphere.
	Consequently,
	\begin{align}\label{SDE.PU.estimate.oldConst}
		\int_0^T &\int_{B_R(0)} \M^R_{\varepsilon\inv}(D b_t)(l_\gamma(x)) v_t(x)\ \mathrm{d}x\mathrm{d}t
		\leq C_v^1\sqrt{|\ln(\varepsilon)|} \\
		&+ \norm[{L^{p}([0,T]\times B_R(0))}]{|D_x b|\mathbbm{1}_{|D_x b| \geq \sqrt{|\ln(\varepsilon)|}}}\left(w_d\ln(1+\varepsilon\inv)C_v^{p'}+C_dC_v^1\right).\notag
	\end{align}
By the de la Vall\'ee--Poussin theorem (see, e.g. \cite[Vol. 1, Theorem 4.5.9]{bogachev2007measure}),  there exists a convex function $G_R:[0,\infty)\to [0,\infty)$ with $G_R(0)=0$ such that
\begin{align}\label{SDE.PU.G_R.superlinear}
	\lim_{r\to \infty} \frac{G_R(r)}{r}= \infty,
\end{align}
and
\begin{align*}
	C_{G_R,|D_x b|} :=\left(\int_0^T \int_{B_R(0)} G_R (|Db_t(x)|^p)\ \mathrm{d}x\mathrm{d}t\right)^{\frac{1}{p}} <\infty.
\end{align*}
Now, using that $(0,\infty) \ni r\mapsto \frac{G_R(r)}{r}$ is non-decreasing (which is due to the convexity of $G_R$ and $G_R(0)=0$), we estimate
\begin{align*}
	\frac{G_R(\sqrt{|\ln(\varepsilon)|}^p)^{\frac{1}{p}}}{\sqrt{|\ln(\varepsilon)|}}&\left(\int_0^T\int_{B_R(0)} |Db_t(x)|^p \mathbbm{1}_{|Db_t|\geq \sqrt{|\ln(\varepsilon)|}}(x)\ \mathrm{d}x\mathrm{d}t\right)^\frac{1}{p}\leq C_{G_R,|D_x b|}.
\end{align*}
So, for $\varepsilon$ sufficiently small such that $G_R(\sqrt{|\ln(\varepsilon)|}^p)>0$ we obtain together with \eqref{SDE.PU.estimate.oldConst} that
\begin{align*}
	\int_0^T &\int_{B_R(0)} \M^R_{\varepsilon\inv}(D b_t)(l_\gamma(x)) v_t(x)\ \mathrm{d}x\mathrm{d}t\\
	&\leq C_v^1\sqrt{|\ln(\varepsilon)|} + C_{\text{max}}\frac{\sqrt{|\ln(\varepsilon)|}\left(\ln(\varepsilon\inv+1) + 1\right)}{G_R(\sqrt{|\ln(\varepsilon)|}^p)^{\frac{1}{p}}},
\end{align*}
where $C_{\text{max}}:=C_{G_R,|DF|}\max\left\{w_dC_v^{p'},C_dC_v^1\right\}$.
Let us define the function $\phi_R:(1,\infty)\to [0,\infty)$ via
\begin{align*}
	\frac{r\ln(r)}{\phi_R(r)}:= C_v^{1}\sqrt{\ln(r)} + C_{\text{max}}\frac{\sqrt{\ln(r)}\left(\ln(r+1) + 1\right)}{G_R(\sqrt{\ln(r)}^p)^{\frac{1}{p}}},
	\end{align*}
	$\text{where }
	r\in A_R:=\{s \in (1,\infty) : G_R(\sqrt{\ln(s)}^p)>0\}$,
and $\phi_R := 0$ on $(1,\infty)\backslash A_R$.
Let us note that (for $r \in A_R$)
\begin{align*}
	\frac{\phi_R(r)}{r\ln(r)} \leq \frac{1}{C_v^{1}\sqrt{\ln(r)}} \to 0, \text{ as $r\to \infty$},
\end{align*}
and 
\begin{align*}
	\frac{\phi_R(r)}{r}= \frac{1}{\frac{C_v^{1}}{\sqrt{\ln(r)}}+C_{\text{max}}\frac{\sqrt{\ln(r)}}{G_R(\sqrt{\ln(r)}^p)^{\frac{1}{p}}}\frac{\ln(r+1) +1}{\ln(r)}}.
\end{align*}
Since
\begin{align*}
	(1,\infty)\ni r \mapsto \frac{\ln(r+1) +1}{\ln(r)}
\end{align*}
is decreasing and converging to $1$, as $r\to \infty$,
 $(1,\infty)\ni r\mapsto \frac{\phi_R(r)}{r}$ is increasing with $\frac{\phi_R(r)}{r} \to \infty$, as $r\to \infty$.
Hence, $\phi_R$ is super-linear of modest growth.
This concludes the proof.
\end{proof}
\begin{proof}[Proof of Theorem \ref{SDE.PU.theorem.mainresult}]
The assertion follows from Theorem \ref{SDE.PU.theorem.mainresult.abstract}, Proposition \ref{SDE.PU.proposition.SobolevMLR} and Lemma \ref{SDE.PU.lemma.SobolevCoeffSatOSL}.
\end{proof}
\section{Proofs for Section \ref{section.mainResults.strongSolutionsSDE}}\label{section.SDE.applications.SDE}
The aim of this section is to prove Theorem \ref{SDE.application.SDE.theorem.figalli}, Theorem \ref{SDE.application.SDE.theorem.lebrislion} and Theorem \ref{SDE.application.SDE.theorem.ams} stated in Section \ref{section.mainResults.strongSolutionsSDE}. In the first subsection we will recall several known results on the existence of solutions to degenerate \eqref{FPKE}. In the second subsection, we will finally combine the results from the first subsection with the pathwise uniqueness results from Section \ref{section.SDE.pathwiseUniqueness} via Theorem \ref{SDE.mainresult.abstract} to obtain probabilistically strong solutions to the associated SDEs.
\subsection{Literature: Existence of solutions to degenerate \eqref{FPKE}}\label{section.SDE.applications.SDE.1}
 In this section, we will recall some prominent results regarding existence of solutions to degenerate \eqref{FPKE} from \cite{figalli2008existence} in the case of bounded coefficients, and from \cite{lebris2008existence} and \cite{bogachev2015FPKE} in the case of unbounded coefficients.
Recall that $\beta=\beta(a,b)$ is defined as in as in \eqref{SDE.results:beta}.
 
\subsubsection{Example for bounded coefficients $b,\sigma$}
The  following result is essentially taken from \cite[Theorem 4.12]{figalli2008existence}.
\begin{theorem}\label{SDE.application.FPKE.theorem.figalli}
	Let $a:[0,T]\times\rd \to S_+(\rd)$ and $b:[0,T]\times \rd \to \rd$ be bounded Borel measurable functions with $\divv\beta \in L^1_{\mathrm{loc}}([0,T]\times\rd)$ and ${(\divv\beta)^- \in L^1([0,T];L^\infty(\rd))}$.
	Let $v \in (L^1\cap L^\infty)(\rd)$ such that $v \geq 0$ a.e.
	Then there exists an $L^1$-solution ${(u_t)_{t\in [0,T]}}$ to \eqref{FPKE} such that $\left.u\right|_{t=0}=v$, $u_t \geq 0$ a.e. for all $t\in [0,T]$, and  ${u \in L^\infty([0,T];L^1(\rd))\cap L^\infty([0,T];L^\infty(\rd))}$. 
\end{theorem}
\subsubsection{Example for unbounded coefficients $b,\sigma$}
The following result is essentially a special case of \cite[Theorem 6.7.4]{bogachev2015FPKE}.\begin{theorem}\label{SDE.application.FPKE.theorem.AMS}
	Let $p,p' \in (1,\infty)$ such that $\nicefrac{1}{p}+\nicefrac{1}{p'}=1$.
Let ${b \in L^p([0,T]\times \rd;\rd)}$ and $a \in L^p([0,T]\times\rd;S_+(\rd))$ such that 
	$a(t,\cdot) \in W^{1,p}(O)$, $\sup_{t\in [0,T]}\norm[W^{1,p}(O)]{a(t,\cdot)}<\infty$ for every Ball $O\subseteq \rd$. Furthermore, 
	assume that $\divv \beta \in L^1_{\mathrm{loc}}([0,T]\times \rd)$ such that $(\divv \beta)^- \in L^1([0,T];L^\infty(\rd))$. 
	
	Let $v \in L^1(\rd)\cap L^{p'}(\rd)$.
	Then there exists an $L^1$-solution $(u_t)_{t\in [0,T]}$ to \eqref{FPKE} with $\left.u\right|_{t=0}=v$ such that $u \in L^\infty([0,T];L^1(\rd))\cap L^\infty([0,T];L^{p'}(\rd))$.
	Furthermore, if $v \geq 0 $ a.e. then $u_t$ can be chosen to be nonnegative almost everywhere for all $t\in [0,T]$.
\end{theorem}
\begin{proof}
By \cite[Theorem 6.7.4]{bogachev2015FPKE}, there exists a solution $(\mu_t)_{t\in (0,T)}$ to \eqref{FPKE} with $\left.\mu\right|_{t=0}=v$ (in the sense of Definition \ref{FPKE.definition.solution.measure}) such that $\mu \in L^\infty([0,T];L^{p'}(\rd))$.
The nonnegativity of $(\mu_t)_{t\in (0,T)}$ follows from the scheme of the proof of \cite[Theorem 6.7.4]{bogachev2015FPKE}:

	The solution $(\mu_t)_{t\in (0,T)}$ is constructed as a weak-* limit in $L^\infty([0,T];L^{p'}(\rd))$ of a sequence of classical solutions $(u^k_t)_{t\in [0,T]}$ \eqref{FPKE} with respect to the operator $L(a_k,b_k)$ and $\left.u^k\right|_{t=0}=v^k$, where, for $k \in \LN$, $a^k$ and $b^k$ are suitable smooth approximations of $a$ and $b$, respectively, $a^k$ nondegenerate diffusion matrices and $v^k$ smooth compactly supported functions such that $v^k\to v$ in $L^{p'}(\rd)$, as $k\to \infty$.
	Clearly, all these classical solutions $(u^k_t)_{t\in [0,T]}$ are positivity preserving in the sense that $v^k \geq 0$ implies $u^k_t \geq 0$ for all $t\in [0,T]$.
	In the case that $v \geq 0$ a.e., we may choose $v^k$ to be nonnegative. From here, the nonnegativity of $(\mu_t)_{t\in (0,T)}$ follows.
	From Proposition \ref{FPKE.solution.massconservation} we conclude that $\mu \in L^\infty([0,T];L^1(\rd))$ (actually $\mu_t(\rd) = \nu(\rd)$ for a.e. $t\in (0,T)$).
	The existence of a narrowly continuous version $(\tilde{\mu}_t)_{t\in (0,T)}$ of $(\mu_t)_{t\in (0,T)}$ with $\left.\tilde{\mu}\right|_{t=0}=v$ follows from Proposition \ref{FPKE.solution.vagueVersion}. Since $\mu \in L^\infty([0,T];L^{p'}(\rd))$, it is an easy argument to see that $\tilde{\mu}_t$ is absolutely continuous with respect to Lebesgue measure for all $t\in [0,T]$. Hence, the assertion follows.
	\end{proof}

The next result is essentially taken from \cite[Section 7, Proposition 5]{lebris2008existence} (or \cite[Theorem 9.8.1]{bogachev2015FPKE}). See also the generalisation \cite{luo2013BV}, where they could replace the condition $\beta\in L^1([0,T];W^{1,1}_{\mathrm{loc}}(\rd;\rd))$ by $\beta\in L^1([0,T];\text{BV}_{\mathrm{loc}}(\rd;\rd))$ in the below theorem.
\begin{theorem}
\label{SDE.application.FPKE.theorem.lebrislion}
	Assume that the following condition on $\sigma$ and $\beta$ are fulfilled:
	\begin{align}
		\sigma \in L^2([0,T];W^{1,2}_{\mathrm{loc}}(\rd;\RR^{d\times d_1})),\ \ \beta\in L^1([0,T];W^{1,1}_{\mathrm{loc}}(\rd;\rd)) \label{SDE.application.lebrisLion.degenerate.cond1}\\
		\divv \beta \in L^1([0,T];L^\infty(\rd)), \ \ \frac{|\beta|}{1+|x|} \in L^1([0,T];L^1(\rd)) + L^1([0,T];L^\infty(\rd))\label{SDE.application.lebrisLion.degenerate.cond2}\\
		\frac{\sigma}{1+|x|} \in L^2([0,T];L^2(\rd;\RR^{d\times d_1})) + L^2([0,T];L^\infty(\rd;\RR^{d\times d_1})).\label{SDE.application.lebrisLion.degenerate.cond3}
	\end{align}
	Let $v \in L^1(\rd)\cap L^\infty(\rd)$.
	Then there exists a unique $L^1$-solution $(u_t)_{t\in [0,T]}$ to \eqref{FPKE} with respect to the operator $L({\sigma\sigma^*\slash 2},b)$ such that $\left.u\right|_{t=0}=v$ and $u\in\mathscr{L}^\sigma_{1,\infty}$, where
	\begin{align*}
	\mathscr{L}^\sigma_{1,\infty}:= \{ \eta \in L^\infty([0,T];(L^1\cap L^\infty)(\rd)) : \sigma^*\nabla \eta \in L^2([0,T];L^2(\rd;\RR^{d_1})\}.
\end{align*}
	Furthermore, if $v \geq 0 $ a.e. then $u_t$ can be chosen to be nonnegative almost everywhere for all $t\in [0,T]$.
\end{theorem}
\begin{proof}
	In \cite[Section 7, Proposition 5]{lebris2008existence}, the authors showed the existence and uniqueness of a solution $(\mu_t)_{t\in (0,T)}$ such that $\left.u\right|_{t=0}=v$ and $\mu \in \mathscr{L}^\sigma_{1,\infty}$ in the sense of Definition \ref{FPKE.definition.solution.measure}. The existence of a narrowly continuous version $(\tilde{\mu}_t)_{t\in [0,T]}$ of $(\mu_t)_{t\in (0,T)}$ follows from Proposition \ref{FPKE.solution.vagueVersion}. Since $\mu \in L^\infty([0,T];L^{\infty}(\rd))$, it is an easy argument to see that $\tilde{\mu}_t$ is absolutely continuous with respect to Lebesgue measure for all $t\in [0,T]$.
	This shows the first part of the assertion.
	The argument for the last part is similar to the one in the proofs of Theorem \ref{SDE.application.FPKE.theorem.figalli} and Theorem \ref{SDE.application.FPKE.theorem.AMS} and will therefore be omitted.
	\end{proof}

\subsection{Proofs of Theorem \ref{SDE.application.SDE.theorem.figalli}, Theorem \ref{SDE.application.SDE.theorem.ams} and Theorem \ref{SDE.application.SDE.theorem.lebrislion}}
\label{section.SDE.applications.SDE.2}
In this subsection we prove the main applications mentioned in the introduction to this chapter.

\begin{proof}[Proof of Theorem \ref{SDE.application.SDE.theorem.figalli}]
Let us first note that by Proposition \ref{FPKE.solution.massconservation}, the $L^1$-solution $(u_t)_{t\in [0,T]}$ with $\left.\mu\right|_{t=0}=v$ provided by Theorem \ref{SDE.application.FPKE.theorem.figalli} is a $\PPPP_0$-solution to \eqref{FPKE}.
Furthermore, $P^{\infty,loc}_{v}$-pathwise uniqueness for \eqref{SDE} is implied by Theorem \ref{SDE.PU.theorem.mainresult}.
Therefore, the assertion follows from Theorem \ref{SDE.mainresult.abstract}.
\end{proof}
\begin{proof}[Proof of Theorem \ref{SDE.application.SDE.theorem.lebrislion}]
The argument is similar to the one in the proof of Theorem \ref{SDE.application.SDE.theorem.figalli}. Here, the assertion follows from Theorem \ref{SDE.application.FPKE.theorem.lebrislion}, Theorem \ref{SDE.PU.theorem.mainresult} via Theorem \ref{SDE.mainresult.abstract}.
\end{proof}
\begin{proof}[Proof of Theorem \ref{SDE.application.SDE.theorem.ams}]
The argument is similar to the one in the proof of Theorem \ref{SDE.application.SDE.theorem.figalli}. Here, the assertion follows from Theorem \ref{SDE.application.FPKE.theorem.AMS}, Theorem \ref{SDE.PU.theorem.rz10} via Theorem \ref{SDE.mainresult.abstract}.
\end{proof}

\section{Proof for Section \ref{section.mainResults.uniquenessFPKE}}\label{section.SDE.applications.FPKE}
In \cite{roeckner2010weakuniqueness}, R\"ockner and Zhang developed a procedure to obtain weak uniqueness results for degenerate Fokker--Planck--Kolmogorov equations, employing the superposition principle by Figalli and a suitable pathwise uniqueness result (see Theorem \ref{SDE.PU.theorem.rz10}).
The essence of their uniqueness result can be boiled down to the following proposition and remark below.
With the help of the updated superposition principle provided by Theorem \ref{SDE.theorem.superpositionRoe} and the pathwise uniqueness results from Section \ref{section.SDE.pathwiseUniqueness}, we can update \cite[Theorem 1.1]{roeckner2010weakuniqueness} in terms of Theorem \ref{SDE.application.FPKE.theorem.uniqueness}, which is proved below.
\begin{proposition}\label{SDE.FPKE.uniqueness.theorem.abstract}
	Let $\nu \in \PPPP(\rd)$ and $a:={\sigma\sigma^*\slash 2}$.
	Let $\mathscr{L}_{\nu}$ be a set of families $(\mu_t)_{t\in [0,T]}\subseteq \PPPP(\rd)$ such that $t\mapsto \mu_t$ is vaguely continuous, $\mu_0=\nu$ for which \eqref{SDE.superposition.conditionRoe} holds.
	Assume that $P^{\mathscr{L}_{\nu}}$-weak uniqueness holds for \eqref{SDE}, where
	\begin{align*}
		P^{\mathscr{L}_{\nu}}:=\{Q \in \PPPP(C([0,T];\rd)): (Q\circ \pi_t\inv)_{t\in [0,T]} \in \mathscr{L}_{\nu}\}.
	\end{align*}
	Then there is at most one solution $(\mu_t)_{t\in [0,T]}$ to \eqref{FPKE} such that $(\mu_t)_{t\in [0,T]} \in \mathscr{L}_{\nu}$.
\end{proposition}
\begin{proof}
	Let $(\mu^1_t)_{t\in [0,T]}, (\mu^2_t)_{t\in [0,T]} \in 	\mathscr{L}_{\nu}$ be two solution to \eqref{FPKE}.
Then, for $i\in\{1,2\}$, there exists a $P_{(\mu^i_t)}$-weak solution $(X^i,W^i)$ to \eqref{SDE} by Theorem \ref{SDE.theorem.superpositionRoe} and Proposition \ref{SDE.proposition.martingaleSolutionEQweaksolution}.
Hence,
$\mu^1_t = \law{X^1(t)}=\law{X^2(t)}=\mu^2_t$ for all $t\in [0,T]$.
This finishes the proof.
\end{proof}
\begin{remark}\label{SDE.FPKE.uniqueness.remark.abstract}
	Note that in the above situation $P^{\mathscr{L}_{\nu}}$-weak uniqueness is implied by $P^{\mathscr{L}_{\nu}}$-pathwise uniqueness, see Remark \ref{SDE.definition.remark.uniqStrongSolutionImpliesWeakUniq}.
\end{remark}
\begin{proof}[Proof of Theorem \ref{SDE.application.FPKE.theorem.uniqueness}]
Let $(\mu^1_t)_{t\in [0,T]}, (\mu^2_t)_{t\in [0,T]} \in 	\mathscr{L}^{p',\bar{p}'}_{\nu}$ be two solutions to \eqref{FPKE}.
Note that by assumption $a^{ij}, b^i \in L^1([0,T]\times\rd;\mu^k_t\mathrm{d}t)$, $1\leq i,j \leq d$, $k \in \{1,2\}$.
Hence, by Proposition \ref{FPKE.solution.massconservation}, $\mu_t^k (\rd)=\nu(\rd)$ for all $\in [0,T]$, $k \in \{1,2\}$. Without loss of generality, we may therefore assume $\nu(\rd)>0$ from now on.
We set
	\begin{align*}
		\tilde{\mu}^i_t := \frac{\mu_t^i}{\nu(\rd)},\ \ t\in [0,T].
	\end{align*}
	Note that $(\mu^1_t)_{t\in [0,T]}, (\mu^2_t)_{t\in [0,T]}\in \mathscr{L}_\nu$ are probability solutions to \eqref{FPKE}, where
	\begin{align*}
		\mathscr{L}_{\nu/\nu(\rd)} := \{(\mu_t)_{t\in [0,T]} \subseteq \PPPP(\rd) : t\mapsto \mu_t \text{ is narrowly cont.}, \mu_0 = \nu/\nu(\rd), \mu \in (L^p\cap L^{p'}_{\mathrm{loc}})([0,T]\times\rd)\}.
	\end{align*}
	Let $P^{\mathscr{L}_{\nu\slash\nu(\rd)}}$ be defined as in Proposition \ref{SDE.FPKE.uniqueness.theorem.abstract}. Now $P^{\mathscr{L}_{\nu\slash\nu(\rd)}}$-pathwise uniqueness is proved for $p\in [1,\infty]$ in Theorem \ref{SDE.PU.theorem.rz10} under the condition \ref{SDE.PU.condition.monotonicity}. 
	If $p=1$ and we assume the alternative Sobolev condition on $b$ and $\sigma$, respectively, we refer to Theorem \ref{SDE.PU.theorem.mainresult}.
	Therefore, the assertion follows via Proposition \ref{SDE.FPKE.uniqueness.theorem.abstract} and Remark \ref{SDE.FPKE.uniqueness.remark.abstract}.
\end{proof}
\section*{Appendix -- On the solutions to \eqref{FPKE}}
In \cite[Chapter 6]{bogachev2015FPKE} Bogachev, Krylov, R\"ockner and Shaposhnikov introduced the following general notion of measure-valued solution to \eqref{FPKE}.
\begin{definition}\label{FPKE.definition.solution.measure}
	We say that $(\mu_t)_{t\in (0,T)}$ is a solution to \eqref{FPKE} if  $t\mapsto |\mu_t|(A)\in L^1((0,T))$ for every Borel measurable set $A\subset\rd$ with compact closure in $\rd$, $b^i,a^{ij} \in L^1_{\mathrm{loc}}((0,T)\times \rd;\mu)$, $1\leq i,j\leq d$, 
 and for all $\varphi \in C_c^\infty((0,T)\times \rd)$ there exists $J_\varphi \subseteq (0,T)$ with $\lambda^1(J_\varphi)=T$ such that
			\begin{align}\label{FPKE.definition.solution.measure:1}
				\int_0^T \int_\rd (\partial_t\varphi + L_{t} \varphi)\ \mathrm{d}\mu_t\mathrm{d}t=0.
			\end{align}
	Furthermore, let $\nu \in \mathcal{M}_{\mathrm{loc}}(\rd)$.
	We say that $\mu$ has initial condition $\nu$, denoted $\left.\mu\right|_{t=0}=\nu$, if for each $\varphi \in C_c^\infty(\rd)$ there exists a Lebesgue measurable set $J_\varphi \in (0,T)$ with $\lambda^1(J_\varphi)=T$ such that
\begin{align}\label{FPKE.definition.solution.measure:2}
	\int_\rd \varphi\ \mathrm{d}\mu_t \to \int_\rd \varphi\ \mathrm{d}\nu, \text{ for $t\to 0$ with $t\in J_\varphi$.} 
\end{align}
\end{definition}
\begin{remark}\label{FPKE.definition.solution.measure:remark}
	In fact (see \cite[Proposition 6.1.2, Remark 6.1.3]{bogachev2015FPKE}), under the stronger conditions in Definition \ref{FPKE.definition.solution}, the notion of solution for \eqref{FPKE} in Definition \ref{FPKE.definition.solution} and Definition \ref{FPKE.definition.solution.measure} are the same, i.e. \eqref{FPKE.definition.solution.measure:1} and \eqref{FPKE.definition.solution.measure:2} are true if and only if \eqref{FPKE.definition.solution:1} is true.
	\end{remark}
	
	Furthermore, we have the following conservation of mass result, which can essentially be found in \cite[Remark 2.7]{figalli2008existence} (see also \cite[Proposition C.1.3, Remark C.1.4]{grube2022thesis}).
	\begin{proposition}\label{FPKE.solution.massconservation}
		Let $\nu\in \mathcal{M}_b(\rd)$. Let $(\mu_t)_{t\in (0,T)}\subseteq \mathcal{M}_b(\rd)$ be a solution to \eqref{FPKE} with $\left.\mu\right|_{t=0}=\nu$, with respect to $b^i, a^{ij} \in L^1([0,T]\times \rd;|\mu_t|\mathrm{d}t)$.
		Then $\mu_t(\rd) = \nu(\rd)$ for a.e. $t\in (0,T)$.
		Furthermore, if $(\mu_t)_{t\in (0,T)}\subseteq \mathcal{M}_+(\rd)$ and is a priori not necessarily a subset of $\mathcal{M}_b(\rd)$, then the same assertion holds.
		If $(\mu_t)_{t\in (0,T)}$ is additionally vaguely continuous, then $\mu_t(\rd) = \nu(\rd)$ for all $t\in (0,T)$.
	\end{proposition}
	The following proposition is essentially a special case of \cite{rehmeier2022nonlinearFlow}, where the result is proved even for nonlinear Fokker--Planck equations. This proposition ensures that under quite general conditions one can find a vaguely continuous version of a solution to \eqref{FPKE}.
	\begin{proposition}\label{FPKE.solution.vagueVersion}
		Let $\nu \in \mathcal{M}_+(\rd)\cap \mathcal{M}_b(\rd)$. Let $(\mu_t)_{t\in (0,T)}\subseteq \mathcal{M}_+(\rd) \cap \mathcal{M}_b(\rd)$ be a solution to \eqref{FPKE} with $\left.\mu\right|_{t=0}=\nu$. Furthermore, assume that there exists $C>0$ such that $|\mu_t|\leq C$ for all $t\in [0,T]$.
		 Then there exists a vaguely continuous version $(\tilde{\mu})_{t\in [0,T]}$ of $({\mu})_{t\in (0,T)}$ such that $\tilde{\mu}_t=\mu_t$ for a.e. $t\in (0,T)$ and $\tilde{\mu}_0 = \nu$.
	\end{proposition}
	\begin{proof}
		The assertion has been proved for $(\mu_t)_{t\in (0,T)}$ being a family of subprobability measures and $\nu$ being a subprobability measure in \cite[Lemma 2.3]{rehmeier2022nonlinearFlow}.
		However, the same proof works here as well.	\end{proof}
	\begin{remark}
		If, in the situation of Proposition \ref{FPKE.solution.vagueVersion}, $\tilde{\mu}_t(\rd)=\nu(\rd)$ for all $t\in [0,T]$, then it is standard to see that $(\tilde{\mu}_t)_{t\in [0,T]}$ is narrowly continuous.
	\end{remark}

\textbf{Acknowledgements:}
	The author would like to thank Michael R\"ockner for valuable discussions.
Furthermore, the author would like to thank Lukas Wresch for a useful discussion on Remark \ref{SDE.PU.remark.SIcounterexample}.
Moreover, the support by the German Research Foundation (DFG) through the IRTG 2235 and the CRC 1283 is gratefully acknowledged.
\printbibliography
\end{document}